\documentclass[12pt]{amsart} 
\usepackage[mathscr]{eucal} 
\usepackage{amsmath,amsfonts} 
\parskip=\smallskipamount 
\newtheorem{theorem}{Theorem}[section] 
\newtheorem{proposition}[theorem]{Proposition} 
\newtheorem{corollary}[theorem]{Corollary} 
\newtheorem{lemma}[theorem]{Lemma} 
\theoremstyle{definition} 
 
\newtheorem{example}[theorem]{Example}

\newtheorem{remark}[theorem]{Remark} 
\newtheorem{remarks}[theorem]{Remarks}
\theoremstyle{fancyproclaim}

\makeatletter 
\@addtoreset{equation}{section} 
\makeatother 
\renewcommand\theequation{\arabic{section}.\arabic{equation}} 
 
\newcommand{\CC}{{\mathbb C}} 
\newcommand{\NN}{{\mathbb N}}

\newcommand{\RR}{{\mathbb R}} 
\newcommand{\FF}{{\mathbb F}}
 
\newcommand{\cA}{{\mathcal A}} 
\newcommand{\cB}{{\mathcal B}} 
\newcommand{\cC}{{\mathcal C}} 
\newcommand{\cD}{{\mathcal D}}

\newcommand{\cH}{{\mathcal H}} 
\newcommand{\cI}{{\mathcal I}}
\newcommand{\cJ}{{\mathcal J}} 
\newcommand{\cK}{{\mathcal K}} 
\newcommand{\cL}{{\mathcal L}} 
\newcommand{\cM}{{\mathcal M}}

\newcommand{\cS}{{\mathcal S}} 
\newcommand{\cT}{{\mathcal T}} 
 
\newcommand{\cV}{{\mathcal V}} 
\newcommand{\cW}{{\mathcal W}}

\newcommand{\dom}{\operatorname{Dom}} 
 
\newcommand{\Ra}{\Rightarrow} 
\newcommand{\LRa}{\Leftrightarrow} 
\newcommand{\ran}{\operatorname{Ran}} 
 
\newcommand{\ra}{\rightarrow}

\newcommand{\tr}{\operatorname{tr}} 
\let\phi=\varphi 
\newcommand{\iac}{\mathrm{i}} 
\renewcommand{\ker}{\operatorname{Ker}} 
 
\newcommand{\de}{\mathrm{d}}
\newcommand{\emath}{\mathrm{e}}

\renewcommand{\Re}{\operatorname{Re}}

\newcommand{\nr}[1]{\vspace{0.1ex}\noindent\hspace*{12mm}\llap{\textup{(#1)}}} 
 
\begin{document} 
\title{On Propagation of Fixed Points of Quantum Operations and Beyond}\thanks{Work supported by a grant of the Romanian 
National Authority for Scientific Research, CNCS  UEFISCDI, project number
PN-II-ID-PCE-2011-3-0119.}
 
 \date{\today}
  
\author[A. Gheondea]{Aurelian Gheondea} 
\address{Department of Mathematics, Bilkent University, 06800 Bilkent, Ankara, 
Turkey, \emph{and} Institutul de Matematic\u a al Academiei Rom\^ane, C.P.\ 
1-764, 014700 Bucure\c sti, Rom\^ania} 
\email{aurelian@fen.bilkent.edu.tr \textrm{and} A.Gheondea@imar.ro} 

\begin{abstract} We show that some abstract results on propagation of fixed points for completely
positive maps on $C^*$-algebras provide a natural approach
to unify recent Noether type theorems on the equivalence of symmetries
with conservation laws for dynamical systems of Markov processes, of
quantum operations, and of quantum stochastic maps. In addition, we obtain some new 
Noether type
theorems, provide examples and counter-examples, and extend most of the existing results
with characterisations in terms of dual infinitesimal generators of the corresponding 
strongly continuous one-parameter semigroups.
\end{abstract} 

\maketitle 

\section{Introduction}

In view of the celebrated theorem of E.~Noether \cite{Noether} on the equivalence of symmetries 
and conservation laws, J.C.~Baez and B.~Fong \cite{BaezFong} considered similar questions 
within the framework of "stochastic mechanics", in the sense of \cite{BaezBiamonte}, for the 
dynamics of Markov processes. Letting 
$\{U(t)\}_{t\geq 0}$ be a (classical) dynamical stochastic system (this is called a 
Markov semigroup in \cite{BaezFong}), they show that the operator of multiplication with an 
observable $O$ commutes with $U_t$ for all $t\geq 0$, an analogue for a symmetry, 
if and only if both its expected value $\langle O,U_tf\rangle$ and the expected value of its square 
$\langle O^2,U_tf\rangle$ are constant in time for every state $f$ (probability distribution), 
an analogue for a conservation law. Considering the variance 
$\langle O^2,f\rangle - \langle O,f\rangle^2$, for $f$ an arbitrary state, 
the latter is equivalent with both its expected value 
and its variance (or standard deviation) are constant in time for every state. 
The appearance of the variance makes a difference 
when compared to the classical Noether's Theorem. It is one of our aim to show
that, when viewing this from the perspective of the approach of \cite{AriasGheondeaGudder},
similar facts have been observed previously in closely related mathematical
problems on irreversible dynamical quantum systems, 
e.g.\ as in S.~Albeverio and R.~H\o egh-Krohn, 
\cite{AlbeverioHoeghKrohn}, E.B.~Davies \cite{Davies}, D.~Evans \cite{Evans},
A.~Frigerio and M.~Verri \cite{FrigerioVerri}, and E.~St\o rmer \cite{Stormer}, to quote a few.
 
More precisely, letting $\cA=\{A_n\}_{n\in\NN}$ be a sequence of positive 
operators in $\cB(\cH)$, for some Hilbert space $\cH$, such that $\sum_{n=1}^\infty A_n=I$ 
one considers the quantum operation $\Phi_\cA$, in the Schr\"odinger picture and the
L\"uders form, defined by
\begin{equation}\label{e:pac}
\Phi_\cA(T)=\sum_{n=1}^\infty A_n^{1/2}TA_n^{1/2},\quad T\in\cB_1(\cH),
\end{equation}
hence, $\Phi_\cA$ is a completely positive and trace preserving linear map on 
$\cB_1(\cH)$. 
Note that in the Heisenberg picture its dual $\Phi^\sharp_\cA$ has the same formal expression as 
in \eqref{e:pac} and that $\Phi^\sharp_\cA$ is a unital normal completely positive linear map on 
$\cB(\cH)$. The equivalence of assertions (i) and (ii) in the 
following proposition was obtained as Corollary~3.4 in \cite{AriasGheondeaGudder}, 
while the equivalence with assertions (iii) and (iv) is clear.

\begin{proposition}\label{p:fipa} Let $\Phi_\cA$ be the unital 
quantum operation in the Schr\"odinger picture as in \eqref{e:pac}, its dual $\Phi_\cA^\sharp$
in the Heisenberg picture, and let $B\in\cB(\cH)$ be a fixed point of $\Phi_\cA^\sharp$. 
The following assertions are equivalent:
\begin{itemize}
\item[(i)] $B$ commutes with all operation elements $A_1,A_2,\ldots$ of $\Phi_\cA$.
\item[(ii)] $B^*B$ and $BB^*$ are fixed points of $\Phi_\cA^\sharp$.
\item[(iii)] The whole unital $C^*$-algebra generated by $B$ is fixed by $\Phi_\cA^\sharp$.
\item[(iv)] The whole von Neumann algebra generated by $B$ is fixed by $\Phi_\cA^\sharp$.
\end{itemize}
\end{proposition}
Note that Proposition~\ref{p:fipa} implies that, a selfadjoint 
operator $B\in\cB(\cH)$ commutes with all operation elements of $\cA$ if and only if both 
$B$ and $B^2$ are fixed points of $\Phi^\sharp_\cA$, hence
a characterisation of exactly the same type with that obtained in 
the Noether type theorem of  \cite{BaezFong}.
There are important differences between these two results, notably the latter condition on
the square of the observable is necessary even in the finite state space case for the classical 
Markov processes, cf.\ the example at page 3 in \cite{BaezFong}, while for 
the L\"uders operation it is not, cf.\ Theorem~3.5 in \cite{AriasGheondeaGudder}. 

The result in \cite{BaezFong} has been put in a setting of dynamical quantum 
systems by J.E.~Gough, T.S.~Ra\c tiu, and O.G.~Smolyanov in \cite{GoughRatiuSmolyanov}. 
More precisely, let $\cT=\{\cT_t\}_{t\geq 0}$ denote a dynamical system in the 
Schr\"odinger picture, that is, a norm
continuous semigroup of completely positive trace-preserving linear maps on the trace-class 
$\cB_1(\cH)$ for some fixed Hilbert space $\cH$, for which the infinitesimal generator $M$ 
takes the form, cf.\ \cite{Lindblad}, \cite{GoriniKossakowskiSudarshan},
\begin{equation}\label{e:ig}
M(S)=\sum_{k}(L_kSL_k^*-\frac{1}{2} SL_k^*L_k-\frac{1}{2}L_k^*L_kS)+\iac[S,H],\quad 
S\in\cB_1(\cH),
\end{equation}
for a collection of operators $L_k\in\cB(\cH)$, $k=1,2,\dots$, and a selfadjoint operator 
$H\in\cB(\cH)$. The \emph{constants} of  $\cT$ are the operators 
$A\in\cB(\cH)$ such that $\tr((\cT_t\rho)A)=\tr(\rho A)$ for all density operators 
$\rho\in\cD(\cH)$ and all $t\geq 0$. Transferring to the Heisenberg picture, one 
considers the dual semigroup $\{\cJ_t\}_{t\geq 0}$ acting in $\cB(\cH)$ whose set of fixed points,
that is, all $A\in\cB(\cH)$ such that $\cJ_t(A)=A$ for all $t\geq 0$, 
coincides with the set of constants of $\cT$. The main result 
in \cite{GoughRatiuSmolyanov} says that, under the technical assumption of existence of 
a stationary strictly positive density operator, the set of
constants of the quantum dynamical system
$\{\cT_t\}_{t\geq 0}$, which coincides with the set of fixed points of $\{\cJ_t\}_{t\geq 0}$, is a 
von Neumann algebra and it coincides with the commutant 
$\{H,L_k,L_k^*\mid k=1,2,\ldots\}^\prime$. In their formulation, an analogue of the second condition
on the square of the observable as in \cite{BaezFong} does not show up and, another aim of our
article is to show that this happens because it is obscured by the technical assumption of 
existence of a stationary strictly positive density operator.

Within the same circle of ideas as in \cite{BaezFong} and \cite{GoughRatiuSmolyanov}, 
K.~Bartoszek and W.~Bartoszek \cite{Bartoszek2} recently considered a noncommutative version 
of dynamical stochastic system, more precisely, a strongly 
continuous semigroup $\{S_t\}_{t\geq 0}$ of stochastic maps with respect to some Hilbert space 
$\cH$, that is, 
trace-preserving positive linear maps on the trace-class $\cB_1(\cH)$, and an one-element 
measurement operator $M_{A^{1/2}}$, for some positive operator $A\in\cB(\cH)$, 
where $M_{A^{1/2}}(T)=A^{1/2}TA^{1/2}$. In this setting, they obtain several equivalent
characterisations to the compatibility
(commutation) of the dynamical stochastic system  $\{S_t\}_{t\geq 0}$ with the quantum 
measurement $M_{A^{1/2}}$: for example, one of these equivalent characterisations
refers to $A$ and $A^2$ being fixed by
the dual semigroup $\{S^\sharp_t\}_{t\geq 0}$ and a second one refers to the commutation of
the infinitesimal generator $\mathfrak{s}$ of $\{S_t\}_{t\geq 0}$ with $M_{A^{1/2}}$. The approach
used in \cite{Bartoszek2} combines the probability theory methods as in \cite{BaezFong} with
operator theoretical methods. There are some important questions left unanswered in 
\cite{Bartoszek2}: for example, to which extent
the additional condition of $A^2$ to be 
fixed by the dual semigroup $\{S^\sharp_t\}_{t\geq 0}$ as well is really necessary? It is another aim 
of our article to provide an answer to this question.

In this article we show that the $C^*$-algebraic dilation theoretical approach as in 
\cite{AriasGheondeaGudder} unifies all the results in \cite{BaezFong}, 
\cite{GoughRatiuSmolyanov}, and \cite{Bartoszek2} under a common framework of propagation
of fixed points for completely positive maps. 
In addition, for each of the noncommutative Noether type theorems considered in 
\cite{GoughRatiuSmolyanov} and \cite{Bartoszek2}, we provide examples
and counter-examples that clarify the necessity of the second order extra conditions imposed,
obtain some new Noether type theorems, and extend most of the existing results
with characterisations in terms of dual infinitesimal generators of strongly continuous 
one-parameter semigroups. Actually, we show that the abstract results obtained in theorems 
\ref{t:unu} and \ref{t:doi} short-cuts completely the probabilistic tools in the proofs of the 
main results in \cite{BaezFong} 
and \cite{Bartoszek2}, while in the case of \cite{GoughRatiuSmolyanov} they reveal
what happens in case the technical assumption of existence of 
a stationary strictly positive density operator is removed.

We briefly describe the contents of this article. Section~\ref{s:pr} contains most of the 
technical results that are needed in this article.
Firstly, we obtain Theorem~\ref{t:unu}
that shows that for a unital linear map $\Phi$ on a $C^*$-algebra $\cA$ 
that is completely positive when restricted to the unital $C^*$-algebra generated by an 
element $a\in\cA$, the fixation of $a$, $a^*a$, and $aa^*$ propagates to the whole unital $C^*$-subalgebra generated by $a$. 
Since, according to a classical theorem of Stinespring \cite{Stinespring}, 
positivity on commutative $C^*$-algebras implies complete positivity, if $a$ is a normal 
element the same conclusions as in Theorem~\ref{t:unu} can be obtained for positive unital maps, 
as in  Corollary~\ref{c:unu}. These results are obtained through a classical result on multiplicative 
domains of M.-D.~Choi \cite{Choi}.
Subsection~\ref{ss:fppmcsa}
provides an equivalent 
characterisation of the set of fixed points for a $w^*$-continuous semigroup of
$w^*$-continuous operators in terms of the null space of its 
$w^*$-infinitesimal generator, which is used in all three cases considered in 
sections~\ref{s:dssm},~\ref{s:dmp}, and \ref{s:cdqs}, in order to obtain characterisations in terms
of the infinitesimal generators of the dual (Markov) semigroups. We think that Theorem~\ref{t:dual}
is most likely known but we could not find a reference for it. 
Subsection~\ref{ss:etvna}
provides a semigroup version of Theorem~2.4 in 
\cite{AriasGheondeaGudder}, more precisely, letting
$\mathbf{\Phi}=\{\Phi_t\}_{t\geq 0}$ be a 
$w^*$-continuous semigroup of $w^*$-continuous, unital, and completely positive maps on a von
Neumann algebra $\cM$, by an ergodic theoretical approach we show that
the set of joint fixed points $\cM^{\mathbf{\Phi}}$ is the range of a completely
positive, unital, and idempotent map $\Psi\colon\cM\ra\cM$. This fact is the technical tool 
to be used, in conjunction with some classical results on injective von Neumann algebras, 
in clarifying the question whether the additional condition on $A^*A$ and $AA^*$ 
in Theorem~\ref{t:dqsc} is necessary, for infinite dimensional Hilbert spaces, 
by adapting the counter-example from \cite{AriasGheondeaGudder} to the semigroup setting.

In Section~\ref{s:dssm} we provide different proofs for the main results of \cite{Bartoszek2}, 
on compatibility of
one-element quantum measurements with stochastic maps in both the discrete and continuous 
dynamical systems cases, more precisely, 
we show that these results can be obtained directly from Corollary~\ref{c:unu}. 
Example~\ref{ex:stochastic} shows, by means of the 
transpose map with respect to a fixed orthonormal basis, that the set of stochastic maps
that are not quantum operations is quite large.
In the case of a continuous dynamical stochastic system, we additionally find two more equivalent characterisations of the 
compatibility of one-parameter semigroups of stochastic maps 
$\{\Psi_t\}_{t\geq 0}$ with an one-element quantum measurement operator $M_{A^{1/2}}$
in terms of the dual infinitesimal generator $\psi^\sharp$: one by the commutation of $M_{A^{1/2}}$ 
with $\psi^\sharp$ and the second by the fact that $\psi^\sharp$ annihilates both $A$ and $A^2$.

In Section~\ref{s:dmp} we consider the setting of dynamical systems of 
classical Markov processes as in  
\cite{BaezFong} and show how the Noether type theorems obtained in that paper
can be naturally recovered under our approach. In the discrete semigroup case, we point out 
additional equivalent characterisations through the dual semigroup while, in the case of a strongly 
continuous semigroup, we obtain additional equivalent characterisations in terms of 
infinitesimal generators, dual semigroups and dual infinitesimal generators. 

In Section~\ref{s:cdqs} we consider a slightly more general setting of 
dynamical quantum systems, when compared to that used in 
\cite{GoughRatiuSmolyanov}, by replacing the operator
norm continuity of the one-parameter semigroup 
with strong continuity, and reorganise most of it in a rather different fashion and 
obtain new results. Firstly, we consider discrete quantum semigroups 
for which we obtain Noether type theorems with 
respect to left and right multiplication by arbitrary bounded operators. Note that, due to the fact
that dynamical quantum systems consist of completely positive maps only, these results go 
beyond multiplication operators with normal operators, a restriction that seems difficult to overcome 
in the case of dynamical stochastic systems as in \cite{Bartoszek2}.
Then, we point out a scale of sets of
constants $\cC^\Psi\supseteq\cC^\Psi_2\supseteq\cC^\Psi_{\mathrm{p}}
\supseteq\cC^\Psi_{\mathrm{c}} \supseteq\cC^\Psi_{\mathrm{w}}$, for $\Psi$ 
a dynamical quantum system (either discrete or continuous), 
and discuss their relation: we show that all these sets but  $\cC^\Psi$ 
coincide and they make a von Neumann algebra, while the first order set of constants $\cC^\Psi$ 
is the largest one and only under special conditions, as the 
existence of a stationary faithful state, coincides with the other sets of constants, 
equivalently, is a von Neumann algebra. For strongly continuous
quantum semigroups, additional characterisations in terms of the infinitesimal generators and 
dual infinitesimal generators are obtained. In Theorem~\ref{t:dqscc} we
show that, in any infinite dimensional and separable Hilbert space, there
exists norm continuous quantum semigroups for which the set of constants is not a von Neumann 
algebra, equivalently, it is not stable under multiplication. This result clarifies also the question
why the extra condition on $A^2$ to be a fixed point is necessary, in general,
in the infinite dimensional noncommutative (quantum) case.

In Appendix we provide a modern proof of Theorem~\ref{t:choi}  as a consequence of the Stinespring's 
Dilation Theorem, following \cite{BrownOzawa}, that shows that the $C^*$-algebraic 
abstract results we rely upon have a dilation theoretical character. 

A few words about terminology. We have used the same names "stochastic" and, respectively, 
"Markov" for both the commutative (classical) case as in Section~\ref{s:dmp}, 
and the noncommutative (quantum) case as in Section~\ref{s:dssm},
hoping that there will be no danger of confusion. This way, we left the notions of quantum
stochastic and, respectively, quantum Markov referring to the case of quantum operations in the 
Schr\"odinger picture and, respectively, in the Heisenberg picture, following the terminology already 
established in quantum physics, see \cite{FagnolaRebolledo} and \cite{GoughRatiuSmolyanov}.

We thank Marius D\u ad\^arlat for drawing our attention to the proof of M.-D.~Choi's Theorem 
in \cite{BrownOzawa} obtainable solely from the Stinespring's Dilation Theorem 
and for many other useful discussions on these topics, to Radu Purice for clarifying some
aspects from \cite{GoughRatiuSmolyanov}, and to Carlo Beenakker for indicating 
\cite{ChruscinskiKossakowski} and \cite{Horodecki4}
as sources on the significance of the transpose map in quantum information theory.

\section{Preliminary Results}\label{s:pr}

\subsection{Propagation of Fixed Points in $C^*$-Algebras}\label{ss:fppmcsa}

Let $\cA$ and $\cB$ be $C^*$-algebras with unit. A linear map $\Phi\colon\cA\ra\cB$
is \emph{positive} if $\Phi(a)\geq 0$ for all $a\in\cA^+$, where $\cA^+=\{x^*x\mid x\in\cA\}$ 
denotes the cone of positive elements in $\cA$. Any positive map is selfadjoint, in the sense
that $\Phi(a^*)=\Phi(a)^*$ for all $a\in\cA$, and bounded, more precisely, 
according to the Russo--Dye Theorem, $\|\Phi\|=\|\Phi(e)\|$, where by $e$ we
denote the unit of $\cA$.

Given an arbitrary natural number $n$, we consider
the $C^*$-algebra $M_n(\cA)$ of all $n\times n$ matrices with entries in $\cA$, organized as a 
$C^*$-algebra in a canonical way, e.g.\ by identifying it with the $C^*$-algebra $\cA\otimes M_n$. 
This gives rise to the $n$-th order \emph{amplification map} 
$\Phi_n\colon M_n(\cA)\ra M_n(\cB)$ defined by
\begin{equation} \Phi_n(A)=\left[ \Phi(a_{i,j})\right]_{i,j=1}^n,\quad A=[a_{i,j}]_{i,j=1}^n\in M_n(\cA).
\end{equation} 
$\Phi$ is called \emph{$n$-positive} if $\Phi_n$ is positive. $\Phi$ is called 
\emph{completely positive} if it is $n$-positive for all $n\in\NN$.

Given $\cA$ a $C^*$-algebra with unit, a closed linear subspace $\cS$ of $\cA$ is called an
\emph{operator system} if it is stable under the adjoint operation $a\mapsto a^*$ and contains
the unit of $\cA$. Note that any operator system is linearly generated by the cone of all its
positive elements. Also, for any linear map $\Psi\colon \cS\ra\cB$, for $\cB$ an arbitrary 
$C^*$-algebra, the definitions of positive map, $n$-positive map, and completely positive map,
as defined before, make perfectly sense. More generally, these definitions make sense if $\cS$
is assumed to be stable under the adjoint operation only.

For an arbitrary linear map $\Phi\colon\cA\ra\cB$, the set
\begin{equation}\label{e:md}
\cM_\Phi=\{a\in\cA\mid \Phi(a^*a)=\Phi(a)^*\Phi(a)\mbox{ and }\Phi(aa^*)=\Phi(a)\Phi(a^*)\}
\end{equation}
is called the \emph{multiplicative domain} of $\Phi$. If $\Phi$ is unital then $\cM_\Phi$
contains the unit of $\cA$.

We start with the following theorem, due to M.-D.~Choi \cite{Choi}; it is worth observing that 
assertion (2) is actually a property of propagation of multiplicativity which motivates the name of 
$\cM_\Phi$. The Schwarz Inequality was first obtained in a special case 
by R.V.~Kadison in \cite{Kadison}, that's why sometimes it is called the Kadison--Schwarz 
Inequality. Following \cite{BrownOzawa},
we provide a short and modern proof of this theorem in Appendix, 
which also points out its dilation 
theory substance, as a consequence of the Stinespring's Dilation Theorem \cite{Stinespring}.

\begin{theorem} \label{t:choi}
Let $\Phi\colon\cA\ra\cB$ be a contractive completely positive map. Then:

\nr{1} \emph{(The Schwarz Inequality)} $\Phi(a)^*\Phi(a)\leq \Phi(a^*a)$ for all $a\in\cA$.

\nr{2} \emph{(The Multiplicativity Property)} Let $a\in\cA$. Then: 
\begin{itemize}
\item[(i)] $\Phi(a^*a)=\Phi(a)^*\Phi(a)$ if and only
if $\Phi(ba)=\Phi(b)\Phi(a)$ for all $b\in\cA$.
\item[(ii)] $\Phi(aa^*)=\Phi(a)\Phi(a)^*$ if and only if $\Phi(ab)=\Phi(a)\Phi(b)$ for all $b\in\cA$.
\end{itemize}

Consequently,
\begin{equation}\label{e:mda}
\cM_\Phi=\{a\in\cA\mid \Phi(ab)=\Phi(a)\Phi(b),\ \Phi(ba)=\Phi(b)\Phi(a),\mbox{ for all }b\in\cA\}.
\end{equation}

\nr{3} The multiplicative domain $\cM_\Phi$ defined at \eqref{e:md} is a $C^*$-subalgebra of $\cA$
and it coincides with the largest $C^*$-subalgebra $\cC$ of $\cA$ such that 
$\Phi|_\cC\colon\cC\ra \cB$ is a $*$-homomorphism.
\end{theorem}

Actually, the Schwarz Inequality is true under the more general condition that $\Phi$ is $2$-positive, 
while the Multiplicativity 
Property holds for  $4$-positive maps: see also \cite{Paulsen}.

We are interested in fixed points of positive maps between $C^*$-algebras. Given a $C^*$-algebra 
$\cA$ with unit $e$, 
let $\Phi\colon\cA\ra\cA$ be a linear map that is unital and positive. We consider the set of the \emph{fixed points} of $\Phi$
\begin{equation}\label{e:afi}
\cA^\Phi=\{a\in\cA\mid \Phi(a)=a\},
\end{equation}
of all fixed points of $\Phi$ and it is easy to see that $\cA^\Phi$ is an operator system. 
Another set of interest is the \emph{bimodule domain}
\begin{equation}\label{e:iafi}
\cI(\Phi)=\{a\in\cA\mid \Phi(ab)=a\Phi(b),\ \Phi(ba)=\Phi(b)a,\mbox{ for all }b\in\cA\},
\end{equation}
which is a $C^*$-subalgebra of $\cA$ containing the unit $e$. Clearly,
\begin{equation}\cI(\Phi)\subseteq\cA^\Phi\cap\cM_\Phi.\end{equation}
On the other hand, if $\Phi$ is completely positive and contractive, 
by Theorem~\ref{t:choi}.(2) we have 
\begin{equation}\cA^\Phi\cap\cM_\Phi=\cI_\Phi.\end{equation}

As shown in \cite{AriasGheondeaGudder}, even for the very particular case of a L\"uders operation 
$\Phi$ on $\cB(\cH)$, where $\cB(\cH)$ denotes 
the von Neumann algebra of all bounded operators on a Hilbert space $\cH$, 
in general we cannot expect that the set of fixed points of $\Phi$ coincides with its bimodule 
domain. In the following we consider a related question: given a unital positive map 
$\Phi\colon\cA\ra\cA$, we want to see whether the quality of an element $a\in\cA$
of being fixed by $\Phi$ 
propagates to the whole $C^*$-algebra $C^*(e,a)$.
In view of Proposition~\ref{p:fipa}, it is not a surprise that
this question is related to the concept of multiplicative domain, that is, imposing 
$a^*a,aa^*\in\cA^\Phi$ and a certain  "locally complete positivity" condition on $\Phi$ as well.

\begin{theorem}\label{t:unu}
Let $\cA$ be a $C^*$-algebra with unit $e$, let 
$\Phi\colon\cA\ra\cA$ be a unital linear map, and let $a\in\cA$ be such that 
$\Phi|_{C^*(e,a)}\colon C^*(e,a)\ra\cA$ is completely positive. 
The following assertions are equivalent:
\begin{itemize}
\item[(i)] $a,a^*a,aa^*\in\cA^\Phi$, that is, $\Phi(a)=a$, $\Phi(a^*a)=a^*a$, and $\Phi(aa^*)=aa^*$.
\item[(ii)] $a\in\cA^\Phi\cap\cM_\Phi$, that is, $\Phi(a)=a$, $\Phi(a^*a)=\Phi(a)^*\Phi(a)$, and 
$\Phi(aa^*)=\Phi(a)\Phi(a)^*$.
\item[(iii)] $\Phi|_{C^*(e,a)}$ 
has the Bimodule Property, that is, $\Phi(ba)=\Phi(b)a$ and
$\Phi(ab)=a\Phi(b)$ for all $b\in C^*(e,a)$.
\item[(iv)] $C^*(e,a)\subseteq\cA^\Phi$, that is, $\Phi(b)=b$ for all $b\in C^*(e,a)$.
\end{itemize}
\end{theorem}

\begin{proof} (i)$\Ra$(ii). By assumptions it follows
\begin{equation*}
\Phi(a^*a)=a^*a=\Phi(a)^*\Phi(a),\quad \Phi(aa^*)=aa^*=\Phi(a)\Phi(a)^*,
\end{equation*} hence, $a\in\cA^\Phi\cap\cM_\Phi$.

(ii)$\LRa$(iii). By assumption and Theorem~\ref{t:choi}.(2), $\Phi|_{C^*(e,a)}$ 
has the Bimodule Property, hence 
$\Phi(xa)=\Phi(x)\Phi(a)=\Phi(x)a$ and $\Phi(ax)=\Phi(a)\Phi(x)=a\Phi(b)$ for all $x\in C^*(e,a)$.

(iii)$\Ra$(iv). By assumption and
using a straightforward induction argument, it follows that, for any $n\in\NN_0$, we have
\begin{equation}\label{e:bmp}
\Phi(xa^n)=\Phi(x)a^n,\quad \Phi(a^nx)=a^n\Phi(x),\quad x\in C^*(e,a),
\end{equation}
and, since $\Phi$ is selfadjoint, we have 
$\Phi(a^*)=\Phi(a)^*=a^*$, hence
\begin{equation}\label{e:bmpa}
\Phi(xa^{*n})=\Phi(x)a^{*n},\quad \Phi(a^{*n}x)=a^{*n}\Phi(x),\quad x\in C^*(e,a).
\end{equation}
From \eqref{e:bmp} and \eqref{e:bmpa}, by a straightforward induction argument, it follows that for 
any monomial $p$ in two noncommutive variables $X$ and $Y$
\begin{equation*}
p(X,Y)=X^{i_1}Y^{j_1}\cdots X^{i_m}Y^{j_m},\quad i_1,\ldots,j_m\in\NN_0,\ j_1,\ldots,j_m\in\NN_0,\ 
m\in\NN,
\end{equation*}
it follows that
\begin{equation}\label{e:pexy}
\Phi(p(a,a^*))=p(a,a^*),
\end{equation} where $p(a,a^*)\in\cA$ is the element obtained by formally
replacing $X$ with $a$ and $Y$ with $a^*$. Then, by linearity, it follows that \eqref{e:pexy} is true
for  any complex polynomials $p$ in two noncommutative variables $X$ and $Y$ hence,
since the collection of all elements of form $p(a,a^*)$
is dense in $\cA$ and $\Phi|_{C^*(e,a)}$ is continuous, assertion (ii) follows.

(iv)$\Ra$(i). This implication is clear.
\end{proof}

As an application of Theorem~\ref{t:unu} we record the special case of a normal element $a$, 
that is, $a^*a=aa^*$, when the condition of "locally complete 
positivity" follows from the condition of positivity.

\begin{corollary}\label{c:unu}
Let $\Phi\colon\cA\ra\cA$ be a linear map which is positive and unital, and let $a\in\cA$ be 
a normal element. The following assertions are equivalent:
\begin{itemize}
\item[(i)] $\Phi(a)=a$ and $\Phi(a^*a)=a^*a$.
\item[(ii)] $\Phi(b)=b$ for all $b\in C^*(e,a)$.
\end{itemize}
\end{corollary}

\begin{proof} Only the implication (i)$\Ra$(ii) requires a proof. To see this, since $a$ is normal it
follows that $C^*(e,a)$ is a commutative $C^*$-algebra hence $\Phi|_{C^*(e,a)}\colon C^*(e,a)
\ra\cA$ is completely positive, see \cite{Stinespring}, and we can apply Theorem~\ref{t:unu}.
\end{proof}

One of the intrinsic deficiency of Theorem~\ref{t:unu} is referring to the fact that we do not know
that $a$ has the Bimodule Property on the whole $C^*$-algebra $\cA$. 
This deficiency is remedied for the case of quantum operations, in
the Heisenberg picture, due to the overall complete positivity property.

\begin{theorem}\label{t:doi}
Let $\cA$ be a $C^*$-algebra with unit $e$, let 
$\Phi\colon\cA\ra\cA$ be a unital completely positive linear map, and let $a\in\cA$. 
The following assertions are equivalent:
\begin{itemize}
\item[(i)] $a,a^*a,aa^*\in\cA^\Phi$, that is, $\Phi(a)=a$, $\Phi(a^*a)=a^*a$, and $\Phi(aa^*)=aa^*$.
\item[(ii)] $a\in\cA^\Phi\cap\cM_\Phi$, that is, $\Phi(ax)=a\Phi(x)$ and $\Phi(xa)=\Phi(x)a$ for all 
$x\in\cA$.
\item[(iii)] $C^*(e,a)\subseteq\cA^\Phi$, that is, $\Phi(b)=b$ for all $b\in C^*(e,a)$.
\end{itemize}
\end{theorem}

The proof of this theorem follows the same line of argumentation as in the proof of 
Theorem~\ref{t:unu} and we omit it.

\subsection{Fixed Points of $w^*$-Continuous One-Parameter Semigroups}\label{ss:fpw}
Let $X$ be a Banach space. We consider a strongly continuous one-parameter 
semigroup $\{\Psi_t\}_{ t\geq 0}$ of linear bounded operators on $X$, that is,
\begin{itemize}
\item[(i)] $\Psi_t\colon X\ra X$ is a bounded linear operator for all $t\geq 0$.
\item[(ii)] $\Psi_s\Psi_t=\Psi_{s+t}$, for all $s,t\geq 0$.
\item[(iii)] $\Psi_0=I$.
\item[(iv)] $\RR_+\ni t\mapsto \Psi_t(x)\in X$ is continuous for each $x\in X$.
\end{itemize}
Under these assumptions, from the general theory of one-parameter semigroups, 
e.g.\ see E.~Hille and R.S.~Phillips \cite{HillePhillips}, N.~Dunford and J.T.~Schwartz \cite{DunfordSchwartz}, 
the \emph{infinitesimal generator} $\psi$ 
exists as a densely defined closed operator on $X$, with
\begin{equation}\label{e:gen}
\psi(x)=\lim_{t\ra0+}\frac{\Psi_t(x)-x}{t}=\frac{\de}{\de t}\Psi_t(x)|_{t=0},\quad x\in\dom(\psi),
\end{equation}
and
\begin{equation}\label{e:domgen}
\dom(\psi)=\{x\in X\mid \lim_{t\ra0+}\frac{\Psi_t(x)-x}{t}\mbox{ exists in } X\}. 
\end{equation}

In addition,  e.g.\ see Corollary VIII.1.5 in \cite{DunfordSchwartz}, the limit 
\begin{equation}\label{e:omega}
\omega=
\lim_{t\ra\infty}\log \|\Psi_t\|/t=\inf_{t>0}\log \|\Psi_t\|/t
\end{equation} exists with the \emph{growth bound} $\omega<\infty$ and, 
e.g.\ see Theorem VIII.1.11 in \cite{DunfordSchwartz},
for any complex number $\lambda$ with 
$\Re\lambda>\omega$, the operator $\lambda I-\psi$ has a bounded 
inverse. Also, by the proof of the Hille-Yosida-Phillips Theorem, e.g.\ see Theorem VIII.1.13 in 
\cite{DunfordSchwartz},
 we have
\begin{equation}\label{e:psite}
\Psi_t(x)=\lim_{\lambda\ra\infty} \emath^{-\lambda t}\sum_{n=0}^\infty \frac{(\lambda^2 t)^n
(\lambda I-\psi)^{-n}(x)}{n!},\quad x\in\dom(\psi), \ t\geq 0.
\end{equation}

Recall that $X^\sharp$ denotes the topological dual space of $X$.
For every strongly continuous one-parameter 
semigroup $\{\Psi_t\}_{t\geq 0}$ of bounded linear operators on $X$, the dual one-parameter 
semigroup $\{\Psi_t^\sharp\}_{t\geq 0}$ of bounded linear operators on $X^\sharp$ exists, that is, 
\begin{equation}\label{e:des}
\langle \Psi_t(x),f\rangle=\langle x,\Psi_t^\sharp(f)\rangle,
\quad x\in X,\ f\in X^\sharp,\ t\geq 0,
\end{equation} with the following properties
\begin{itemize}
\item[(i)] $\Psi_t^\sharp\colon X^\sharp\ra X^\sharp$ 
is a linear bounded and $w^*$-continuous operator for all $t\geq 0$.
\item[(ii)] $\Psi_t^\sharp\Psi_s^\sharp=\Psi_{s+t}^\sharp$, for all $s,t\geq 0$.
\item[(iii)] $\Psi_0^\sharp=I$.
\item[(iv)] $\RR_+\ni t\mapsto \Psi_t^\sharp(f)\in X^\sharp$ is $w^*$-continuous for each 
$f\in X^\sharp$.
\end{itemize}
Then, e.g.\ see \cite{Phillips}, 
$\{\Psi_t^\sharp\}_{t\geq 0}$ is a $w^*$-continuous semigroup of operators on 
$X^\sharp$ and hence,
the \emph{$w^*$-infinitesimal generator} $\psi^\sharp$ 
exists as a $w^*$-closed operator on $X^\sharp$, hence a closed operator on $X^\sharp$, with
\begin{equation}\label{e:genp}
\psi^\sharp(f)=w^*\mbox{-}\lim_{t\ra0+}\frac{\Psi_t^\sharp(f)-f}{t}=w^*\mbox{-}
\frac{\de}{\de t}\Psi_t^\sharp(f)|_{t=0},
\end{equation}
and
\begin{equation}\label{e:domgenp}
\dom(\psi^\sharp)=\{f\in X^\sharp\mid w^*\mbox{-}\lim_{t\ra0+}\frac{\Psi_t^\sharp(f)-f}{t}
\mbox{ exists in } X^\sharp\}. 
\end{equation}

The notation we use for $\psi^\sharp$ looks like an abuse but actually it is not: 
by the R.S.~Phillips's Theorem in \cite{Phillips},
\begin{equation}\label{e:degen}
\dom(\psi^\sharp)=\{f\in X^\sharp\mid X\ni f\mapsto\langle x,\psi(f)\rangle\mbox{ is 
continuous }\},
\end{equation}
and
\begin{equation}\label{e:depsi}
\langle \psi(x),f\rangle=\langle x,\psi^\sharp(f)\rangle,\quad 
x\in\dom(\psi),\ f\in\dom(\psi^\sharp),
\end{equation}
hence, the $w^*$-infinitesimal generator $\psi^\sharp$ of the dual $w^*$-continuous
semigroup $\{\Psi^\sharp_t\}_{t\geq 0}$ on $X^\sharp$
is indeed the dual operator of the infinitesimal generator $\psi$ of the strongly continuous 
semigroup $\{\Psi_t\}_{t\geq 0}$ on $X$ and, consequently, the notation for 
$\psi^\sharp$ is fully justified.

In addition, one of the major differences between the two infinitesimal generators 
$\psi$ and $\psi^\sharp$ is that $\dom(\psi^\sharp)$ may not be dense in $X^\sharp$, although it is 
always $w^*$-dense, while $\dom(\psi)$ is always dense in $X$.

\begin{theorem}\label{t:dual}
Let $\{\Psi_t\}_{t\geq 0}$ be a strongly continuous semigroup of operators on a
Banach space $X$, let $\{\Psi_t^\sharp\}_{t\geq 0}$ be the associated dual $w^*$-continuous
semigroup of operators on $X^\sharp$, and $\psi$ and, respectively, $\psi^\sharp$, 
their infinitesimal generators. Considering $f\in X^\sharp$, the following assertions are equivalent:
\begin{itemize}
\item[(i)] $\Psi_t^\sharp(f)=f$ for all real $t\geq 0$.
\item[(ii)] $f\in\ker(\psi^\sharp)$, that is, $f\in\dom(\psi^\sharp)$ and $\psi^\sharp(f)=0$.
\end{itemize}
\end{theorem}

\begin{proof} (i)$\Ra$(ii). This is a clear consequence of \eqref{e:genp} and \eqref{e:domgenp}.

(ii)$\Ra$(i). Let $\lambda>\max\{\omega,0\}$, where $\omega$ is defined as in \eqref{e:omega}. 
Since $\psi^\sharp$ is the dual operator of $\psi$, as in \eqref{e:depsi} and \eqref{e:degen}, and
$\lambda I-\psi$ is boundedly invertible, it follows that $\lambda I-\psi^\sharp$ is boundedly 
invertible, e.g.\ see Theorem 1.5 in \cite{Phillips}. 
Consequently, for any $x\in\dom(\psi)$ and any $g\in X^\sharp$ we have
\begin{align*}
\langle x,\emath^{-\lambda t}\sum_{n=0}^\infty 
\frac{(\lambda^2 t)(\lambda I-\psi^\sharp)^{-n}(g)}{n!} \rangle
& = \langle x,\emath^{-\lambda t}\bigl(\sum_{n=0}^\infty 
\frac{(\lambda^2 t)(\lambda I-\psi)^{-n}}{n!}\bigr)^\sharp(g)\rangle \\
& =\langle \emath^{-\lambda t}\sum_{n=0}^\infty 
\frac{(\lambda^2 t)(\lambda I-\psi)^{-n}(x)}{n!},g\rangle \end{align*}
hence, by \eqref{e:psite} it follows that
\begin{equation}\label{e:lima}
\lim_{\lambda\ra\infty}
\langle x,\emath^{-\lambda t}\sum_{n=0}^\infty 
\frac{(\lambda^2 t)(\lambda I-\psi^\sharp)^{-n}(g)}{n!} \rangle= 
\langle \Psi_t(x),g\rangle.
\end{equation}

On the other hand, from $\psi^\sharp(f)=0$ it follows that 
$(\lambda I-\psi^\sharp)(f)=\lambda f$ hence 
$(\lambda I-\psi^\sharp)^{-1}(f)=\frac{1}{\lambda} f$. By induction we obtain
\begin{equation}\label{e:lan}
(\lambda I-\psi^\sharp)^{-n}(f)=\frac{1}{\lambda^n}f,\quad n\geq 0.
\end{equation}
Consequently, it follows that
\begin{equation*}
\sum_{n=0}^\infty \frac{(\lambda^2 t)^n (\lambda I-\psi^\sharp)^{-n}(f)}{n!}=\sum_{n=0}^\infty
\frac{(\lambda t)^n}{n!}f=\emath^{\lambda t}f,
\end{equation*}
hence, letting $g=f$ in \eqref{e:lima}, it follows that 
\begin{equation*}
\langle x,\Psi_t^\sharp (f)\rangle= \langle\Psi_t(x),g\rangle
= \lim_{\lambda\ra\infty} \langle x,\emath^{-\lambda t}
\emath^{\lambda t}f\rangle = \langle x,f\rangle,
\end{equation*}
and then, since $\dom(\psi)$ is dense in $X$, it follows that $\Psi_t^\sharp(f)=f$ for all $t\geq 0$.
\end{proof}

\subsection{An Ergodic Theorem in von Neumann Algebras}\label{ss:etvna}
We first recall some definitions, in addition to those in Subsection~\ref{ss:fppmcsa}.
Let $\cA$ and $\cB$ be $C^*$-algebras and let $\cV\subseteq \cA$ and $\cW\subseteq\cB$ 
be subspaces. For any linear map $\Phi\colon\cV\ra\cW$ and any natural number $n$, 
the \emph{$n$-th order amplification} $\Phi_n\colon \cV\otimes M_n\ra \cW\otimes M_n$ can 
be defined as $\Phi_n=\Phi\otimes I_n$, where $I_n$ denotes the identity operator on $M_n$. 
Explicitly, by means of the canonical identifications $M_n(\cV)=\cV\otimes M_n$ and 
$M_n(\cW)=\cW\otimes M_n$, this means
\begin{equation}
\Phi_n([v_{i,j}]_{i,j=1}^n)= [\Phi(v_{i,j})]_{i,j=1}^n,\quad [v_{i,j}]_{i,j=1}^n\in M_n(\cV).
\end{equation}
Note that,
by the embeddings $M_n(\cV)\subseteq M_n(\cA)$ and $M_n(\cW)\subseteq M_n(\cB)$, it
follows that $M_n(\cV)$ and, respectively, $M_n(\cW)$ have canonical norms induced by the 
$C^*$-norms on $M_n(\cA)$ and $M_n(\cB)$. Consequently, we can let $\|\Phi_n\|$ denote 
the corresponding operator norm. Clearly,
\begin{equation}
\|\Phi\|=\|\Phi_1\|\leq \|\Phi_2\|\leq \cdots\leq \|\Phi_n\| \leq \|\Phi_{n+1}\|\leq \cdots.
\end{equation}
The map $\Phi$ is called \emph{completely bounded} if 
\begin{equation}
\|\Phi\|_\mathrm{cb}=\sup_{n\geq 1} \|\Phi_n\|<\infty.
\end{equation}
Let $\mathcal{CB}(\cV,\cW)$ denote the vector space of all completely bounded maps 
$\Phi\colon\cV\ra\cW$. Also, such a map $\Phi$ is called \emph{completely contractive} if 
$\|\Phi\|_\mathrm{cb}\leq 1$. A linear map $\Phi\colon\cV\ra\cV$ is called an \emph{idempotent}
if $\Phi^2=\Phi \Phi=\Phi$ and, it is called a \emph{projection} if it is completely contractive 
and idempotent. A subspace $\cV\subseteq\cB(\cH)$, for some Hilbert space $\cH$, 
is called \emph{injective} if there exists a projection $\Phi\colon\cB(\cH)\ra\cB(\cH)$ with 
range equal to $\cV$.

A linear map $\Phi\colon\cA\ra\cA$ is called a \emph{conditional expectation} if it is positive, 
idempotent, and  it has the following bimodule property: 
$\Phi(ar)=\Phi(a)r$ and $\Phi(ra)=r\Phi(a)$, for all $a\in\cA$ and all $r\in\ran(\Phi)$. 
By a classical result 
of J.~Tomyama \cite{Tomiyama}, a $C^*$-algebra 
$\cA\subseteq\cB(\cH)$ is injective if and only if there is a conditional expectation in $\cB(\cH)$ 
with range equal to $\cA$.

For a semigroup $\mathbf{\Phi}=\{\Phi_t\}_{t\geq 0}$ of 
unital, completely positive maps on a $C^*$-algebra $\cM$, 
we consider $\cM^{\mathbf{\Phi}}$ the set of joint
fixed points of $\mathbf{\Phi}$, that is,
\begin{equation}
\cM^\mathbf{\Phi}=\bigcup_{t\geq 0}\cM^{\Phi_t}=\{a\in\cM\mid \Phi_t(a)=a,\mbox{ for all }t\geq 0\},
\end{equation}
see Subsection~\ref{ss:fppmcsa}, which is an operator system, 
and the joint bimodule domain
\begin{align}
\cI(\mathbf{\Phi}) & =\bigcap_{t\geq 0} \cI(\Phi_t)\\
& =\{a\in\cM\mid \Phi_t(ab)=a\Phi_t(b),\ 
\Phi_t(ba)=\Phi_t(b)a,\mbox{ for all }b\in\cA\mbox{ and all } t\geq 0\},\nonumber
\end{align}
which is clearly a $C^*$-subalgebra of $\cM$ and included in  $\cM^{\mathbf{\Phi}}$. In 
case $\cM$ is a von Neumann algebra and each $\Phi_t$ is $w^*$-continuous, 
$\cM^{\mathbf{\Phi}}$ is $w^*$-closed and $\cI(\mathbf{\Phi})$ is a von Neumann subalgebra of 
$\cM$.

\begin{theorem}\label{t:inj}
Let $\cM$ be a von Neumann algebra and $\mathbf{\Phi}=\{\Phi_t\}_{t\geq 0}$ be a 
$w^*$-continuous semigroup of $w^*$-continuous, unital, completely positive maps on $\cM$. 
Then:
\begin{itemize}
\item[(a)] There exists  a completely
positive,  unital, and idempotent map $\Psi\colon\cM\ra\cM$ such that the
set of joint fixed points $\cM^{\mathbf{\Phi}}$ is the range of $\Psi$.
\item[(b)] The following assertions are equivalent:
\begin{itemize}
\item[(i)] $\cM^{\mathbf{\Phi}}$ is stable under multiplication.
\item[(ii)] $\cM^{\mathbf{\Phi}}$ is a von Neumann algebra.
\item[(iii)] $\cM^{\mathbf{\Phi}}=\cI(\mathbf{\Phi})$.
\item[(iv)]  $\Psi$ is a conditional expectation.
\end{itemize}
\item[(c)] If $\cM=\cB(\cH)$ and $\cB(\cH)^{\mathbf{\Psi}}$ is stable under multiplication, then
$\cB(\cH)^{\mathbf{\Psi}}$ is an injective von Neumann algebra.
\end{itemize}
\end{theorem}

\begin{proof} (a) For each real number $t> 0$, let $\Psi_t\colon\cM\ra\cM$ be defined by
\begin{equation}
\Psi_t=\frac{1}{t}\int_0^t \Phi_s\de s.
\end{equation}
The integral converges with respect to the point-$w^*$-topology, that is, for all $a\in\cM$ and all 
$f\in\cM_*$, we have
\begin{equation*}
\langle \Psi_t(a),f\rangle=\frac{1}{t}\int_0^t \langle \Phi_s(a),f\rangle \de s.
\end{equation*} 
It is easy to see that $\Psi_t$ is 
$w^*$-continuous, unital, and completely positive and hence,  by Russo--Dye's Theorem, 
a completely contractive map for each $t> 0$. By the 
Alaoglu's Theorem, the closed unit ball of $\cM$ is $w^*$-compact, hence by Tyhonov's Theorem 
the closed unit ball of 
$\mathcal{CB}(\cM)$ is compact with respect to the point-$w^*$-topology. Consequently, 
considering the sequence $\{\Psi_n\}_{n\in\NN}$, there exists a subsequence 
$\{\Psi_{k_n}\}_{n\in\NN}$ such that
\begin{equation*}
w^*\mbox{-}\lim_{n\ra \infty}\Psi_{k_n}(a)=\Psi(a),\quad a\in\cM,
\end{equation*}
for some linear map $\Psi\colon\cM\ra\cM$. Clearly, $\Psi$ is unital and completely positive.
Let $t\geq 0$ be an arbitrary real number and 
$n\in\NN$ be large enough such that $t\leq n$. Then
\begin{align*}
\Psi_{n}-\Phi_t \Psi_{n} & = \frac{1}{n}\bigl(\int_0^n\Phi_s\de s-\int_0^{n}\Phi_{t+s}\de s\bigr) \\
& =\frac{1}{n} \bigl( \int_0^n\Phi_s\de s-\int_t^{t+n}\Phi_s\de s\bigr) \\
& = \frac{1}{n}\bigl( \int_0^t\Phi_s\de s-\int_n^{t+n}\Phi_s\de s\bigr)
\end{align*}
hence
\begin{equation}\label{e:pan}
\|\Psi_n-\Phi_t\Psi_n\|\leq \frac{1}{n}\bigl( \int_0^t\|\Phi_s\|\de s-\int_n^{t+n}\|\Phi_s\|\de s\bigr)
=\frac{2t}{n}\xrightarrow[n\ra\infty]{} 0.
\end{equation}

On the other hand, using the representation
\begin{equation}
\Phi_t\Psi-\Psi=(\Phi_t\Psi-\Phi_t\Psi_{k_n})+(\Phi_t\Psi_{k_n}-\Psi_{k_n})+(\Psi_{k_n}-\Psi),\quad 
n\in\NN,
\end{equation}
and taking into account that, for all $a\in\cM$, by the defining property of the subsequence 
$(\Psi_{k_n})_{n\in\NN}$, we have
\begin{equation*}
(\Phi_t\Psi-\Phi_t\Psi_{k_n})(a)=\Phi_t(\Psi(a)-\Psi_{k_n}(a))\xrightarrow[n\ra\infty]{w^*} 0,
\end{equation*}
and then of \eqref{e:pan},
it follows that $\Phi_t\Psi=\Psi$, for all $t\geq 0$. 
Similarly we obtain $\Psi\Phi_t=\Psi$ for all $t\geq 0$, hence
\begin{equation}\label{e:pas}
\Phi_t\Psi=\Psi\Phi_t=\Psi,\mbox{ for all }t\geq 0.
\end{equation}
From \eqref{e:pas} we get
\begin{equation*}
\Psi_{k_n}(\Psi(a))=\frac{1}{k_n}\int_0^{k_n}\Phi_s(\Psi(a))\de s=\Psi(a),\quad a\in\cM,\ n\in\NN,
\end{equation*}
and then letting $n\ra\infty$ it follows that $\Psi\Psi=\Psi$, hence $\Psi$ is an idempotent. 
If $a\in\cM^{\mathbf{\Phi}}$ is arbitrary, then $\Psi_{k_n}(a)=a$ for all $n\in\NN$ whence, letting 
$n\ra\infty$ it follows $\Psi(a)=a$. We have proven that $\cM^{\mathbf{\Phi}}\subseteq\ran(\Psi)$.
Since, by \eqref{e:pas}, $\ran(\Psi)\subseteq\cM^{\mathbf{\Phi}}$, we have
$\cM^{\mathbf{\Phi}}=\ran(\Psi)$.

(b) Only the equivalence of (i) and (iv) requires a proof.

Assume firstly that $\cM^{\mathbf{\Phi}}$ is stable under multiplication. By the result at item (a),
it follows that $\ran(\Psi)=\cM^{\mathbf{\Phi}}$ is a von Neumann algebra. Then, for arbitrary 
$a\in\ran(\Psi)$,
\begin{equation*}
\Psi(a)^*\Psi(a)=a^*a=\Psi(a^*a),\quad \Psi(a)\Psi(a)^*=aa^*=\Psi(aa^*),
\end{equation*}
hence, by Theorem~\ref{t:choi}, for any $b\in \cM$ we have
\begin{equation*}
\Psi(ab)=\Psi(a)\Psi(b)=a\Psi(b),\quad \Psi(ba)=\Psi(b)\Psi(a)=\Psi(b)a,
\end{equation*}
consequently $\Psi$ is a conditional expectation.

Conversely, if $\Psi$ is a conditional expectation then $\cM^{\mathbf{\Phi}}=\ran(\Psi)$ is a 
$C^*$-algebra, hence stable under multiplication.

(c) This is now a consequence of the results proven at item (a) and item (b).
\end{proof}

\section{Dynamical Systems of Stochastic/Markov Maps: The Noncommutative Case}
\label{s:dssm}

Let $\cH$ be a Hilbert space, let $\cB(\cH)$ be the von Neumann algebra of all 
bounded linear operators $T\colon \cH\ra\cH$ and let $\cB_1(\cH)$ be the \emph{trace-class}, 
that is, the collection of all 
operators $T\in\cB(\cH)$ subject to the condition $\|T\|_1=\tr(|T|)<+\infty$, where $|T|=(T^*T)^{1/2}$ 
denotes the module of $T$ and $\tr$ denotes the usual normal faithful semifinite \emph{trace} on 
$\cB(\cH)$. Let 
$\cD(\cH)$ denote the set of \emph{states}, or \emph{density operators}, 
with respect to $\cH$, that is, the set of 
all positive elements $\rho\in\cB_1(\cH)$ with $\tr(\rho)=\|\rho\|_1=1$.

A linear map $\Psi\colon \cB_1(\cH)\ra\cB_1(\cH)$ is called 
\emph{stochastic} if it maps states into states, equivalently,
if it is positive, that is, $\Psi(A)\geq0$ for all $A\in\cB_1(\cH)^+$, 
and \emph{trace-preserving}, that is, 
$\tr(\Psi(T))=\tr(T)$ for all $T\in\cB_1(\cH)$. The map $\Psi\colon\cB_1(\cH)\ra\cB_1(\cH)$ 
is called a \emph{quantum operation}, 
if it is completely positive, see Subsection~\ref{ss:fppmcsa} for definition, 
and trace-preserving. Note that, the trace-class $\cB_1(\cH)$ is considered here as a
$*$-subspace of the $C^*$-algebra $\cB(\cH)$ and, consequently, the concept of completely
positive map on $\cB_1(\cH)$ makes perfectly sense. 
Clearly, any quantum operation is a stochastic map. 

We note that the definition of a quantum operation we adopt here is a bit more restrictive than 
usual. In 
quantum information theory they use the term of a \emph{quantum communication channel}, or
briefly a \emph{quantum channel}, for what we call here a quantum operation.

For a fixed Banach space $X$, recall that we denote its topological dual space by $X^\sharp$ 
and the duality map by 
$X\times X^\sharp\ni(x,f)\mapsto \langle x,f\rangle$, see Subsection~\ref{ss:fpw}. 
The topics of this article refer to the 
Banach space $(\cB_1(\cH),\|\cdot\|_1)$ and its topological dual Banach space $(\cB(\cH),\|\cdot\|)$
with the duality map $\cB_1(\cH)\times \cB(\cH)\ni(T,S)\mapsto \langle T,S\rangle=\tr(TS)$, 
e.g.\ see Theorem~19.2 in \cite{Conway}. 
In particular, for a quantum operation $\Psi$ when viewed as a trace-preserving completely positive
map $\Psi\colon\cB_1(\cH)\ra\cB_1(\cH)$, one usually refers to the Schr\"odinger picture, to which the Heisenberg picture is corresponding by duality:
the dual map $\Psi^\sharp \colon\cB(\cH)\ra\cB(\cH)$ is defined by
\begin{equation*}
\langle \Psi(T),S\rangle = \tr(\Psi(T)S)=\tr(T\Psi^\sharp(S))=\langle T,\Psi^\sharp(S)\rangle,\quad 
T\in\cB_1(\cH),\ S\in\cB(\cH),
\end{equation*}
and it is a ultraweakly continuous ($w^*$-continuous) 
completely positive and unital linear map. Similarly, if $\Psi$ is a  
stochastic linear map then its dual $\Psi^\sharp$ is a ultraweakly continuous positive and 
unital linear map on $\cB(\cH)$, called a \emph{ Markov map}.

There are many quantum operations. For example, if $\{A_k\mid k\in\NN\}$ is a collection of 
operators in $\cB(\cH)$ such that $\sum_{k=1}^\infty A_kA_k^*=I$ then the linear map
$\cB_1(\cH)\ni T\mapsto\sum_{k=1}^\infty A_k^*TA_k\in\cB_1(\cH)$ is a quantum operation. 
The following example shows that there exist  
stochastic maps that are not quantum operations. The 
idea of using the transpose map for this kind of examples can be tracked back to
W.B.~Arveson \cite{Arveson1}, \cite{Arveson2}. Stochastic maps that are not quantum
operations, in particular the transpose map, play an important role in entanglement detectors in
quantum information theory, e.g.\ see D.~Chruscinski and A.~Kossakowski 
\cite{ChruscinskiKossakowski}, R.~Horodecki et al. \cite{Horodecki4} and the rich bibliography 
cited there.

\begin{example}\label{ex:stochastic} Let $\cH$ be an arbitrary Hilbert space with dimension at least
$2$, for which we fix an orthonormal basis $\{e_j\}_{j\in\cJ}$. We consider the \emph{conjugation} 
operator $J\colon \cH\ra\cH$ defined by $Jh=\overline h$ where, for arbitrary 
$h=\sum_{j\in\cJ} h_j e_j$, we let $\overline h=\sum_{j\in\cJ}\overline{h}_j e_j$. Then $J$ is 
conjugate linear, conjugate selfadjoint, that is, it has the following property
\begin{equation}\label{e:lap} \langle Jh,k\rangle=\langle Jk,h\rangle,\quad h,k\in\cH,
\end{equation}
isometric, and $J^2=I$.

Further on, let $\tau\colon \cB(\cH)\ra\cB(\cH)$ be defined by $\tau(S)=JS^*J$, for all 
$T\in\cB(\cH)$. It is easy to see that $\tau$ is isometric, that is, 
$\|\tau(S)\|=\|S\|$ for all $S\in\cB(\cH)$, and that $\tau(I)=I$. On the other hand, if $S\in\cB(\cH)^+$
then 
\begin{equation*}
\langle \tau(S)h,h\rangle=\langle JSJh,h\rangle=\langle Jh,SJh\rangle=\langle 
SJh,Jh\rangle\geq 0,\quad h\in\cH,
\end{equation*} hence $\tau$ is positive. Let us also observe that, with respect to the matrix 
representation of operators in $\cB(\cH)$
associated to the orthonormal basis $\{e_j\}_{j\in\cJ}$, $\tau$ is the \emph{transpose 
map}: if $T$ has the matrix representation $[t_{i,j}]_{i,j\in\cJ}$ then $\tau(T)$ has the matrix 
representation $[t_{j,i}]_{j,i\in\cJ}$.

We claim now that $\tau$ leaves $\cB_1(\cH)$ invariant and the corresponding restriction map
$\cB_1(\cH)\ra\cB_1(\cH)$ is stochastic. To see this, we first observe that
if $T\in\cB_1(\cH)^+$ we 
have $\tau(T)\in\cB_1(\cH)^+$, e.g.\ using that $\tau$ is the 
transpose map with respect to the matrix representations of operators in $\cB_1(\cH)$ associated
to the orthonormal basis $\{e_j\}_{j\in\cJ}$, and the definition of the trace in terms of any 
orthonormal basis of $\cH$. Also, $\|\tau(T)\|_1=\tr(\tau(T))=\tr(T)=\|T\|_1$. 
Since any operator $T\in\cB_1(\cH)$ is a linear combination of four positive trace-class operators, 
the claim follows.

Finally, we show that $\tau$ is not completely positive, more precisely, it is not $2$-positive. To see
this, we consider the \emph{matrix units} $\{E_{i,j}\}_{i,j\in\cJ}$, that is, for any 
$i,j\in\cJ$, $E_{i,j}$ denote  
the rank $1$ operator on $\cH$ with $E_{i,j}e_j=e_i$ and $E_{i,j}e_k=0$ for all 
$k\neq j$ and observe that $\tau(E_{i,j})=E_{j,i}$. 
Since $\dim\cH\geq 2$, there exist $i,j\in\cJ$ with $i\neq j$. 
Then, consider the positive finite rank operator in $M_2(\cB_1(\cH))$ defined by
\begin{equation*}E=\left[\begin{matrix} E_{i,i} & E_{i,j} \\ E_{j,i} & E_{j,j}\end{matrix}\right]
\end{equation*}
and observe that 
\begin{equation*} \tau_2(E)=\left[\begin{matrix} \tau(E_{i,i}) & \tau(E_{i,j}) \\ \tau(E_{j,i}) & 
\tau(E_{j,j})\end{matrix}\right]=\left[\begin{matrix} E_{i,i} & E_{j,i} \\ E_{i,j} & E_{j,j}\end{matrix}\right]
\end{equation*} which is not positive, e.g.\ see \cite{Paulsen}, p.~5. Therefore, 
$\tau$ is a stochastic map but not a quantum operation.
\end{example}

\begin{remarks}\label{r:cex} (1) 
By means of the matrix transpose interpretation of 
$\tau$ as in Example~\ref{ex:stochastic}, 
it follows easily that its dual $\tau^\sharp\colon\cB(\cH)\ra\cB(\cH)$ has the same 
formal definition: $\tau(S)=JS^*J$, for all $S\in\cB(\cH)$, and the same matrix transpose 
interpretation with respect to a fixed orthonormal basis of $\cH$.

(2) The stochastic map $\tau$ described in Example~\ref{ex:stochastic} is 
invertible, $\tau^{-1}=\tau$, and antimultiplicative, that is, $\tau(ST)=\tau(T)\tau(S)$ for all 
$S,T\in\cB_1(\cH)$. The same properties are shared by its dual $\tau^\sharp$. 
In particular, both $\tau$ and $\tau^\sharp$ are $*$-antihomomorphisms.
 
(3) In addition to the map $\tau$ described in Example~\ref{ex:stochastic}, 
many other stochastic maps that are not quantum operations can be obtained by considering 
convex combinations of linear maps of type $\tau\circ\Psi$ or $\Psi\circ\tau$, 
where $\Psi$ are quantum operations.
\end{remarks}

\subsection{Discrete Semigroups of  Stochastic/Markov Maps.}\label{ss:dsqso}
From the quantum measurements point of view, given
a quantum operation $\Psi$,
it is of interest to characterise those elements $A\in\cB(\cH)$ with the property that $[\Psi,M_A]=0$,
that is, $\Psi(A^*XA)=A^*\Psi(X)A$ for all $X\in\cB_1(\cH)$, where 
$M_A\colon\cB_1(\cH)\ra\cB_1(\cH)$ denotes the one-element 
measurement, that is, the linear map $M_A(X)=A^*XA$ for all $X\in \cB_1(\cH)$ and the 
commutator is defined as usually 
$[\Phi,\Psi]=\Phi\Psi-\Psi\Phi$. Note that, since $\cB_1(\cH)$ is a two-sided ideal of $\cB(\cH)$, 
$M_A$ can be defined either as a linear map $\cB_1(\cH)\ra\cB_1(\cH)$ or as a linear map 
$\cB(\cH)\ra\cB(\cH)$. Actually, if we consider $M_A\colon\cB_1(\cH)\ra\cB_1(\cH)$
then its dual map $M_A^\sharp \colon\cB(\cH)\ra\cB(\cH)$ is the one-element
measurement map $M_{A^*}$.

A sequence $\{\Psi_n\}_{n\geq 0}$ is called a \emph{discrete  stochastic semigroup}
if
\begin{itemize}
\item[(qs1)] $\Psi_n\colon\cB_1(\cH)\ra\cB_1(\cH)$ is a  stochastic operator for all
integer $n\geq 0$.
\item[(qs2)] $\Psi_{n+m}=\Psi_n\Psi_m$ for all integer $m,n\geq 0$.
\item[(qs3)] $\Psi_0=I$.
\end{itemize}
Clearly, to any  stochastic operator $\Psi$ one associates the discrete
semigroups $\{\Psi^n\}_{n\geq 0}$ and, conversely, any discrete semigroup of  
stochastic operators $\{\Psi_n\}_{n\geq 0}$ is fully determined by $\Psi=\Psi_1$ and 
$\Psi_n=\Psi^n$, for all integer $n\geq 0$. Consequently, the analysis of discrete  
stochastic semigroups pertains to the analysis of one  stochastic operator.

\begin{remark}\label{r:ma} Let $\Psi\colon\cB_1(\cH)\ra\cB_1(\cH)$ be a bounded linear map and 
$A\in\cB(\cH)$. Then $[\Psi,M_A]=0$ if and only if $[\Psi^\sharp,M_{A^*}]=0$.
\end{remark}

The one-element measurement operator $M_A$ is
usually associated to a positive operator $A$. In this case, one rather considers 
the one-element measurement in the L\"uders form $M_{A^{1/2}}$ for some positive operator $A$.
We show that the following theorem, obtained in \cite{Bartoszek2}, can be recovered 
as a rather direct application of Corollary~\ref{c:unu}.

\begin{theorem}[\cite{Bartoszek2}]\label{t:bartoszek}
Let $\Psi$ be a  stochastic map on the Hilbert space $\cH$ and $A\in\cB(\cH)^+$.
The following assertions are equivalent:
\begin{itemize}
\item[(i)] $[\Psi,M_{A^{1/2}}]=0$, that is, $\Psi(A^{1/2}TA^{1/2})=A^{1/2}\Psi(T)A^{1/2}$ for all $T\in\cB_1(\cH)$.
\item[(ii)] $[\Psi^\sharp,M_{A^{1/2}}]=0$, that is, $\Psi^\sharp(A^{1/2}SA^{1/2})=A^{1/2}
\Psi^\sharp(S)A^{1/2}$ for all 
$S\in\cB(\cH)$.
\item[(iii)] $\Psi^\sharp(A)=A$ and $\Psi^\sharp(A^2)=A^2$.
\end{itemize}
\end{theorem}

Before proceeding to the proof of this theorem,
we prove two preliminary results. The first one is essentially
Remark~5.4 in \cite{Bartoszek2} for which we provide a coordinate free proof.

\begin{lemma}\label{l:bartoszek1}
If $E$ is a projection and $C\in\cB_1(\cH)^+$ such that $\tr(C)=\tr(ECE)$ then $C=CE=EC$.
\end{lemma}

\begin{proof} Taking into account that $C^{1/2}EC^{1/2}\leq C$ and that
\begin{equation*}
0\leq \tr(C-C^{1/2}EC^{1/2})=\tr(C)-\tr(C^{1/2}EC^{1/2})=\tr(C)-\tr(ECE)=0,
\end{equation*} it follows that $C=C^{1/2}EC^{1/2}$
hence, 
\begin{equation*}0=C^{1/2}(I-E)C^{1/2}=C^{1/2}(I-E)(I-E)C^{1/2}=((I-E)C^{1/2})^* ((I-E)C^{1/2}),
\end{equation*} 
which implies $(I-E)C^{1/2}=0$ hence $(I-E)C=0$. From here it follows $EC=C$ and then taking 
adjoints we have $CE=C$ as well.
\end{proof}

The second preliminary result is a short-cut of Corollary~5.5, Corollary~5.6, and
Lemma~5.7 in \cite{Bartoszek2}.

\begin{lemma}\label{l:bartoszek2}
Let $\Psi$ be a  stochastic map with respect to a Hilbert space $\cH$ 
and let $E$ be a projection such that $\Psi^\sharp(E)=E$. Then
\begin{itemize}
\item[(i)] $\Psi(ETE)=E\Psi(ETE)=\Psi(ETE)E$ for all $T\in\cB_1(\cH)$.
\item[(ii)] $E\Psi^\sharp(ESE)=\Psi^\sharp(ESE)E=\Psi^\sharp(ESE)$ for all $S\in\cB(\cH)$.
\item[(iii)] $\Psi^\sharp(ESE)=E\Psi^\sharp(S)E$ for all $S\in\cB(\cH)$.
\end{itemize}
\end{lemma}

\begin{proof} (i) It is sufficient to prove this for all $T\in\cB_1(\cH)^+$. With this assumption,
we have
\begin{align*} \tr(E\Psi(ETE)E) & =\tr(E\Psi(ETE)) 
=\langle E,\Psi(ETE)\rangle\\ 
& = \langle \Psi^\sharp(E),ETE\rangle=\langle E, ETE\rangle=\tr(ETE)=\tr(\Psi(ETE)),
\end{align*}
and, consequently, applying Lemma~\ref{l:bartoszek1} for $C=\Psi(ETE)$, the conclusion follows.

(ii) To see this, without loss of generality it is sufficient to assume that $S\in\cB(\cH)^+$ 
is a contraction, that is, $0\leq S\leq I$. Then $0\leq ESE\leq E$ hence 
$0\leq \Psi^\sharp(ESE)\leq \Psi^\sharp(E)=E$, which implies that the range of $\Psi^\sharp(ESE)$ 
is contained in the range of $E$. This implies $E\Psi^\sharp(ESE)=\Psi^\sharp(ESE)$ and then, by
taking adjoints, we have $\Psi^\sharp(ESE)E=\Psi^\sharp(ESE)$ as well.

(iii) Let $T\in\cB_1(\cH)$ and $S\in\cB(\cH)$ be arbitrary. Using assertion (ii) we have
\begin{align*}
\langle \Psi^\sharp(ESE),T\rangle & = \langle E\Psi^\sharp(ESE)E,T\rangle = \langle \Psi^\sharp(ESE),ETE\rangle \\
& = \langle ESE,\Psi(ETE)\rangle = \langle S,E\Psi(ETE)E\rangle \\
\intertext{and then, using assertion (i), we have}
& = \langle S,\Psi(ETE)\rangle =\langle \Psi^\sharp(S),ETE\rangle=\langle E\Psi^\sharp(S)E,T
\rangle,
\end{align*}
hence assertion (iii) follows.
\end{proof}

\begin{proof}[Proof of Theorem~\ref{t:bartoszek}]
(i)$\LRa$(ii). This is a consequence of Remark~\ref{r:ma}.

(ii)$\Ra$(iii). Since $\Psi^\sharp$ is unital it follows that $\Psi^\sharp(A)=\Psi^\sharp(A^{1/2}IA^{1/2})
=A^{1/2}\Psi^\sharp(I)A^{1/2}=A^{1/2}A^{1/2}=A$ and then $\Psi^\sharp(A^2)
=\Psi^\sharp(A^{1/2}AA^{1/2})=A^{1/2}\Psi^\sharp(A)A^{1/2}=A^{1/2}AA^{1/2}=A^2$.

(iii)$\Ra$(ii). Letting $\Psi^\sharp=\Phi$ in Corollary~\ref{c:unu}, it follows that $\Psi^\sharp(S)=S$ 
for all $S\in C^*(I,A)$. Since $\Psi^\sharp$ is $w^*$-continuous, 
by functional calculus with bounded Borel functions on $\sigma(A)$, it follows that 
$\Psi^\sharp(S)=S$ for all $S\in W^*(A)$, the von Neumann algebra generated by 
$A$ in $\cB(\cH)$. In particular, for any spectral projection $E$ of $A$ we have $\Psi^\sharp(E)=E$.
From Lemma~\ref{l:bartoszek2} it follows
\begin{equation}\label{e:pes}
\Psi^\sharp(ESE)=E\Psi^\sharp(S)E,\quad S\in\cB(\cH).
\end{equation}
From here, by the Spectral Theorem for $A$, it follows that
for any function $f$ that is continuous on $\sigma(A)$
we have
\begin{equation}
\Psi^\sharp(f(A)Sf(A))=f(A)\Psi^\sharp(S)f(A),\quad S\in\cB(\cH).
\end{equation} 
Letting $f(t)=g(t)=\sqrt{t}$, $t\in\sigma(A)$, the assertion follows.
\end{proof}

\begin{remarks}\label{r:general}
(1) Under the assumptions of Theorem~\ref{t:bartoszek}, from the proof provided here 
and Remark~\ref{r:ma},
one can easily obtain the following assertions that are mutually equivalent with each of assertions 
(i)--(iii), cf.\ \cite{Bartoszek2}:
\begin{itemize}
\item[(iv)] $[M_E,\Psi]=0$ for any spectral projection $E$ of $A$.
\item[(v)] $[M_E,\Psi^\sharp]=0$ for any spectral projection $E$ of $A$.
\item[(vi)] $[M_{f(A)},\Psi]=0$ for any real function $f$ continuous on $\sigma(A)$.
\item[(vii)] $[M_{f(A)},\Psi^\sharp]=0$ for any real function $f$ continuous on $\sigma(A)$.
\end{itemize}

(2) The mutually equivalent assertions as in Theorem~\ref{t:bartoszek} can be, equivalently,
written in terms of the discrete dynamical system $\{\Psi^n\}_{n\in\NN_0}$:
\begin{itemize}
\item[(i)] $[\Psi^n,M_{A^{1/2}}]=0$ for all $n\in\NN_0$.
\item[(ii)] $[\Psi^{\sharp n},M_{A^{1/2}}]=0$ for all $n\in\NN_0$.
\item[(iii)] $\Psi^{\sharp n}(A)$ and $\Psi^{\sharp n}(A^2)$ do not depend on $n\in\NN_0$.
\end{itemize}
This way, assertions (i) and (ii) are symmetry properties while assertion (iii) is a conservation law.

(3) A natural question related to Theorem~\ref{t:bartoszek} is whether the latter condition 
in item (iii) on $A^2$ being fixed by $\Psi^\sharp$ 
is really necessary for a given  stochastic map $\Psi$. 
It is interesting that, for the transpose
map $\tau$ as in Example~\ref{ex:stochastic} the answer is no. More precisely, let $A\in\cB(\cH)^+$
be a fixed point of $\tau^\sharp$. Since $A$ is positive and taking into account
that $\tau^\sharp$ is antimultiplicative, see Remark~\ref{r:cex}.(2), it follows that 
$\tau^\sharp(A^2)=\tau^\sharp(A)\tau^\sharp(A)=A^2$. However, 
the answer to this question is positive, in general, for quantum operations, see 
Remark~\ref{r:cexq}.(2), and hence for stochastic maps as well.
\end{remarks}

\subsection{Continuous One-Parameter Semigroups of  Stochastic/Markov Maps.}
\label{ss:sopssm}
With notation as in the previous section, we consider a strongly continuous one-parameter 
semigroup $\mathbf{\Psi}=\{\Psi_t\}_{ t\geq 0}$ of stochastic maps with respect to some Hilbert 
space $\cH$.
Under these assumptions, we observe that $\{\Psi_t\}_{t\geq 0}$ is uniformly bounded on 
$\cB_1(\cH)$. Most of the following facts that we briefly recall refer to a particular situation 
of the general theory of one-parameter semigroup theory on Banach spaces, 
e.g.\ see \cite{HillePhillips} and \cite{DunfordSchwartz},
see Subsection~\ref{ss:fpw}.
Given a strongly continuous semigroup $\mathbf{\Psi}=\{\Psi_t\}_{t\geq 0}$ of stochastic maps with respect
to some Hilbert space $\cH$,
the \emph{infinitesimal generator} $\psi$ 
exists as a densely defined closed operator on $\cB_1(\cH)$.
For every strongly continuous one-parameter 
semigroup $\mathbf{\Psi}=\{\Psi_t\}_{ t\geq 0}$ of stochastic maps, the dual one-parameter 
semigroup $\mathbf{\Psi}^\sharp=\{\Psi_t^\sharp\}_{t\geq 0}$ of Markov maps exists, that is, 
\begin{equation}\label{e:dess}
\langle \Psi_t(T),S\rangle=\tr(\Psi_t(T)S)=\tr(T\Psi_t^\sharp(S))=\langle T,\Psi_t^\sharp(S)\rangle,
\ T\in\cB_1(\cH),\ S\in\cB(\cH),\ t\geq 0.
\end{equation} 
Then
$\{\Psi_t^\sharp\}_{t\geq 0}$ is a $w^*$-continuous semigroup of contractions on 
$\cB(\cH)$ and hence,
the \emph{$w^*$-infinitesimal generator} $\psi^\sharp$ 
exists as a $w^*$-closed operator on $\cB(\cH)$, hence a closed operator on $\cB(\cH)$.
The $w^*$-infinitesimal generator $\psi^\sharp$ of the dual $w^*$-continuous
semigroup $\{\Psi^\sharp_t\}_{t\geq 0}$ of Markov maps
is indeed the dual operator of the infinitesimal generator $\psi$ of the strongly continuous 
semigroup $\{\Psi_t\}_{t\geq 0}$ of stochastic maps and, consequently, the notation for 
$\psi^\sharp$ is fully justified.

Also, let us observe that, since $\Psi^\sharp_t(I)=I$, it follows that
\begin{equation}\label{e:iep}
I\in\dom(\psi^\sharp)\mbox{ and }\psi^\sharp(I)=0.
\end{equation}
In addition, one of the major differences between the two infinitesimal generators 
$\psi$ and $\psi^\sharp$ is that $\dom(\psi^\sharp)$ may not be dense in $\cB(\cH)$, although it is 
always $w^*$-dense, while $\dom(\psi)$ is always dense in $\cB_1(\cH)$.

The equivalence of (i)--(iv) in the following theorem has been obtained in \cite{Bartoszek2}. We
add two more equivalent characterisations in terms of the dual infinitesimal generator,
which actually make the proofs simpler.

\begin{theorem}\label{t:bartoszek2}
Let $\mathbf{\Psi}=\{\Psi_t\}_{t\geq 0}$ be a strongly continuous one-parameter semigroup of 
stochastic maps on $\cB_1(\cH)$, $\psi$ its infinitesimal generator, 
and let $A\in\cB(\cH)^+$. With notation as before, the following assertions 
are equivalent:
\begin{itemize}
\item[(i)] $\Psi_t^\sharp(A)=A$ and $\Psi_t^\sharp(A^2)=A^2$ for all $t\geq 0$.
\item[(ii)] $[M_{A^{1/2}},\Psi_t]=0$ for all $t\geq 0$.
\item[(iii)] $[M_{A^{1/2}},\Psi_t^\sharp]=0$ for all $t\geq 0$.
\item[(iv)] $[M_{A^{1/2}},\psi]=0$ that is, for all $T\in\dom(\psi)$ we have $A^{1/2}TA^{1/2}\in
\dom(\psi)$ and $\psi(A^{1/2}TA^{1/2})=A^{1/2}\psi(T)A^{1/2}$.
\item[(v)] $[M_{A^{1/2}},\psi^\sharp]=0$ that is, for all $S\in\dom(\psi^\sharp)$ 
we have $A^{1/2}SA^{1/2}\in\dom(\psi^\sharp)$ and 
$\psi^\sharp(A^{1/2}TA^{1/2})=A^{1/2}\psi^\sharp(T)A^{1/2}$.
\item[(vi)] $A,A^2\in\ker(\psi^\sharp)$, that is, $A,A^2\in\dom(\psi^\sharp)$ and $\psi^\sharp(A)=
\psi^\sharp(A^2)=0$.
\end{itemize} 
\end{theorem}

\begin{proof} The equivalence of the assertions (i), (ii), and (iii) is a straightforward 
consequence of Theorem~\ref{t:bartoszek}.

(ii)$\Ra$(iv). For arbitrary $T\in\dom(\psi)$ and $t\geq 0$,
we have
\begin{align*}
\frac{\Psi_t(A^{1/2}TA^{1/2})-A^{1/2}TA^{1/2}}{t} 
& =\frac{A^{1/2}\Psi_t(T)A^{1/2}-A^{1/2}TA^{1/2}}{t}\\
& =A^{1/2}\frac{\Psi_t(T)-T}{t}A^{1/2}\xrightarrow[t\ra0+]{ } A^{1/2}\psi(T)A^{1/2},
\end{align*} hence $A^{1/2}TA^{1/2}\in\dom(\psi)$ and 
$\psi(A^{1/2}TA^{1/2})=A^{1/2}\psi(T)A^{1/2}$.

(iv)$\Ra$(v). Let $S\in\dom(\psi^\sharp)$. Then, for any $T\in\dom(\psi)$ we have 
$A^{1/2}TA^{1/2}\in\dom(\psi)$ and 
$\psi(A^{1/2}TA^{1/2})=A^{1/2}\psi(T)A^{1/2}$, hence
\begin{equation*}
\langle\psi(T),A^{1/2}SA^{1/2}\rangle=\langle A^{1/2}\psi(T)A^{1/2},S\rangle=
\langle\psi(A^{1/2}TA^{1/2}),S\rangle
\end{equation*}
whence, taking into account of the continuity of the map 
$\cB_1(\cH)\ni T\mapsto A^{1/2}TA^{1/2}\in\cB_1(\cH)$, it follows that 
$A^{1/2}SA^{1/2}\in \dom(\psi^\sharp)$. Consequently,
\begin{align*}
\langle T,\psi^\sharp(A^{1/2}SA^{1/2})\rangle & =\langle\psi(T),A^{1/2}SA^{1/2}\rangle 
=\langle\psi(A^{1/2}TA^{1/2}),S\rangle\\ & =\langle A^{1/2}TA^{1/2},\psi^\sharp(S)\rangle
=\langle T,A^{1/2}\psi^\sharp(S)A^{1/2}\rangle,
\end{align*} hence, 
$\psi^\sharp(A^{1/2}SA^{1/2})=A^{1/2}\psi^\sharp(S)A^{1/2}$.

(v)$\ra$(vi). By \eqref{e:iep} we have $A=A^{1/2}IA^{1/2}\in\dom(\psi^\sharp)$ and 
$\psi^\sharp(A)=A^{1/2}\psi^\sharp(I)A^{1/2}=0$. Then, $A^2=A^{1/2}AA^{1/2}\in\dom(\psi^\sharp)$ 
and $\psi^\sharp(A^2)=A^{1/2}\psi^\sharp(A)A^{1/2}=0$.

(vi)$\Ra$(i). This is a consequence of Theorem~\ref{t:dual}.
\end{proof}

\begin{remarks}\label{r:uta}
(a) Under the assumptions of Theorem~\ref{t:bartoszek2}, each of the assertions (i)--(vi)
is equivalent with each of the following assertions, cf.\ \cite{Bartoszek2}:
\begin{itemize}
\item[(vii)] $\frac{\de}{\de t}\langle \Psi_t(T),A\rangle=\frac{\de}{\de t}\langle \Psi_t(T),A^2\rangle=0$
for all $T\in\cB_1(\cH)$.
\item[(viii)] $\frac{\de}{\de t}\langle \Psi_t(T),A^n\rangle=0$ for all $T\in\cB_1(\cH)$ and all 
$n\geq 0$.
\item[(ix)] For every spectral projection $E$ of $A$ we have $[M_{A^{1/2}},\psi]=0$, that is, for 
any $T\in\dom(\psi)$ we have $ETE\in\dom(\psi)$ and $\psi(ETE)=E\psi(T)E$.
\end{itemize}
The equivalence of assertion (ix) is short-cut in our proof but it is an important step during the proof
provided in \cite{Bartoszek2}. Assertion (vii) is clearly equivalent with assertion (i), while 
assertion (viii) is equivalent with assertion (vii) in view of Corollary~\ref{c:unu}.

(b) A natural question is whether the condition that $A^2$ is a joint fixed point of $\mathbf{\Psi}$, 
as in Theorem~\ref{t:bartoszek2}.(i), is a consequence of the condition that $A$ 
is a joint fixed point of $\mathbf{\Psi}$. The answer is negative, in general, 
and it will be obtained as a consequence of Theorem~\ref{t:dqscc}.
\end{remarks}

\section{Dynamics for Markov Processes: The Real Commutative Case}\label{s:dmp}

In this section we consider the setting of dynamics of Markov processes in the framework 
of "stochastic mechanics" in the sense of \cite{BaezFong} and \cite{BaezBiamonte}. Let $(X;\mu)$
be a $\sigma$-finite measure space. A \emph{probability 
distribution} $p$ is an element in $L^1_\RR(X;\mu)$ which is positive and $\|p\|_1=1$.

An \emph{observable} $O$ is an element in $L^\infty_\RR(X;\mu)$, identified with the operator of 
multiplication $O\colon L^1_\RR(X;\mu)\ra L^1_\RR(X;\mu)$
\begin{equation*}
(Og)(x)=O(x)g(x),\quad g\in L^1_\RR(X;\mu),\ x\in X.
\end{equation*}
The \emph{expected value} of the observable $O$ 
with respect to a probability distribution $g$ is
\begin{equation*}
E(O;g)=\langle O,g\rangle=\int_X O(x)g(x)\de\mu(x),
\end{equation*}
the \emph{variance} of $O$ with respect to $g$ is
\begin{equation*}
V(O;g)=\langle O^2,g\rangle-\langle O,g\rangle^2,
\end{equation*}
while the \emph{standard deviation} of $O$ with respect to $g$ is
\begin{equation*}
\sigma(O;g)=\sqrt{\langle O^2,g\rangle-\langle O,g\rangle^2}.
\end{equation*}

A \emph{stochastic operator} is a bounded linear operator 
$U\colon L^1_\RR(X;\mu)\ra L^1_\RR(X;\mu)$ that
maps probability distributions to probability distributions, equivalently, 
$U$ is \emph{positive}, that is,
\begin{equation*}\mbox{ if }g\in L^1_\RR(X;\mu)\mbox{ and } g\geq 0\mbox{ then }Ug\geq 0,
\end{equation*}
and
\begin{equation*}
\int_X (Ug)(x)\de\mu(x)=\int_Xg(x)\de\mu(x),\quad\mbox{ for all }g\in L^1_\RR(X;\mu).
\end{equation*}
The latter condition can also be written as
\begin{equation*}
\langle 1,Ug\rangle =\langle 1,g\rangle,\quad g\in L^1_\RR(X;\mu).
\end{equation*}

A bounded linear operator $T\colon L^\infty_\RR(X;\mu)\ra L^\infty_\RR(X;\mu)$ is called a 
\emph{Markov map} if it is $w^*$-continuous, 
\emph{positive}, in the sense that for any $f\in L^\infty_\RR(X;\mu)$
with $f\geq 0$ it follows $Tf\geq 0$, and \emph{unital}, that is, $T1=1$. 

Given any bounded linear operator $U\colon L^1_\RR(X;\mu)\ra L^1_\RR(X;\mu)$ there exists its 
\emph{dual operator} $U^\sharp \colon L^\infty_\RR(X;\mu)\ra L^\infty_\RR(X;\mu)$, which is linear 
and bounded, defined by
\begin{align*}
\langle Ug,f\rangle & =\int_X (Ug)(x)f(x)\de\mu(x)=\int_X g(x) (U^\sharp f)(x)\de
\mu(x) \\
& =\langle g,U^\sharp f\rangle,\ f\in L^1_\RR(X;\mu),\ g\in L^\infty_\RR(X;\mu).
\end{align*}
In addition, $U^\sharp$ is $w^*$-continuous.
If $U\colon L^1_\RR(X;\mu)\ra L^1_\RR(X;\mu)$ is a stochastic operator then its dual $U^\sharp\colon 
L^\infty_\RR(X;\mu) \ra L^\infty_\RR(X;\mu)$ is a Markov operator.

\subsection{Discrete Stochastic/Markov Semigroups.}
A \emph{discrete stochastic semigroup} with respect to the measure space $(X;\mu)$
is a sequence $\{U_n\}_{n\geq 0}$  subject to the following
conditions:
\begin{itemize}
\item[(ms1)] $U_n\colon L^1_\RR(X;\mu)\ra L^1_\RR(X;\mu)$ is stochastic for all $n\geq 0$.
\item[(ms2)] $U_{n+m}=U_n U_m$ for all $n,m\geq 0$.
\item[(ms3)] $U_0=I$.
\end{itemize}
Clearly, any discrete stochastic semigroup is of the form
\begin{equation*}
U_n=U^n,\quad n\geq 0,
\end{equation*} where $U=U_1$ is a stochastic operator. Considering the dual
operator $U^\sharp\colon L^\infty_\RR(X;\mu)\ra L^\infty_RR(X;\mu)$, which is actually a Markov 
operator, we can equivalently discuss of \emph{discrete Markov semigroups}.

The equivalence of assertions (i), (ii), (i)$^\prime$, and (ii)$^\prime$
in the following theorem has been obtained in 
\cite{BaezFong}, for which we provide a proof based on the results in Subsection~\ref{ss:fppmcsa}, 
as well as complete their theorem with two more equivalent assertions in terms of duals of
stochastic operators. 

\begin{theorem}\label{t:dms}
Let $(X;\mu)$ be a $\sigma$-finite measure space, $U\colon L^1_\RR(X;\mu)\ra L^1_\RR(X;\mu)$ 
a stochastic 
operator and $O\in L^\infty_\RR(X;\mu)$ an observable. The following assertions are equivalent:
\begin{itemize}
\item[(i)] $[O,U]=0$.
\item[(ii)] For any probability distribution $g$ on $X$ we have $\langle O,Ug\rangle=\langle O,
g\rangle$ and $\langle O^2,Ug\rangle=\langle O^2,g\rangle$.
\item[(i)$^\prime$] $[O,U^n]=0$ for all $n\geq 0$.
\item[(ii)$^\prime$] For any probability distribution $g$ on $X$, the expected values of $O$ 
and $O^2$ with respect to $U^ng$ do not depend on $n\geq 0$.
\item[(i)$^{\prime\prime}$] $[O,U^\sharp]=0$.
\item[(ii)$^{\prime\prime}$] $U^\sharp(O)=O$ and $U^\sharp(O^2)=O^2$.
\end{itemize}
\end{theorem}

\begin{proof} The equivalences (i)$\LRa$(i)$^\prime$, (ii)$\LRa$(ii)$^\prime$, 
(i)$\LRa$(i)$^{\prime\prime}$, and (ii)$\LRa$(ii)$^{\prime\prime}$ are clear. 

(i)$^{\prime\prime}\Ra$(ii)$^{\prime\prime}$. Assume that $[O,U^\sharp]=0$  hence, for any 
$f\in L^\infty_\RR(X;\mu)$ we have $OU^\sharp (f)=U^\sharp (Of)$. 
Letting $f=1$ and taking into
account that $U^\sharp (1)=1$ it follows $U^\sharp(O)=O$ and then letting $f=O$ we have
$U^\sharp (O^2)=OU^\sharp(O)=O^2$.

(ii)$^{\prime\prime}\Ra$(i)$^{\prime\prime}$.
The spaces $L^1_\RR(X;\mu)$ and $L^\infty_\RR(X;\mu)$ are naturally embedded in
$L^1_\CC(X;\mu)$ and, respectively, in $L^\infty_\CC(X;\mu)$. The real stochastic operator $U$ 
can be naturally lifted to a complex stochastic operator 
$U\colon L^1_\CC(X;\mu)\ra L^1_\CC(X;\mu)$. More precisely, since 
\begin{equation*} L^1_\CC(X;\mu)=L^1_\RR(X;\mu)\oplus \iac L^1_\RR(X;\mu),\end{equation*}
we can define $\widetilde U\colon L^1_\CC(X;\mu)\ra L^1_\CC(X;\mu)$ by
\begin{equation*} \widetilde U(g+\iac f)=Ug+\iac Uf,\quad f,g\in L^1_\RR(X;\mu),
\end{equation*}
and observe that $\widetilde U$ has the following two properties:
\begin{equation*}\mbox{ if }g\in L^1_\CC(X;\mu)\mbox{ and } g\geq 0\mbox{ then }\widetilde
Ug\geq 0,
\end{equation*}
and
\begin{equation*}
\int_X (\widetilde Ug)(x)\de\mu(x)=\int_Xg(x)\de\mu(x),\quad\mbox{ for all }
g\in L^1_\CC(X;\mu).
\end{equation*}
Then $\widetilde U^\sharp\colon L^\infty_\CC(X;\mu)\ra L^\infty_\CC(X;\mu)$ is unital and positive.
Since $L^\infty_\CC(X;\mu)$ is a commutative $C^*$-algebra, $\widetilde U^\sharp$ 
is completely positive, cf.\ \cite{Stinespring}.

On the other hand, the observable $O$ can be naturally viewed as a real valued function in 
$L^\infty_\CC(X;\mu)$ and, if $U^\sharp(O)=O$ and $U^\sharp(O^2)=O^2$, 
it follows that $\widetilde U^\sharp(O)=O$ and $\widetilde U^\sharp(O^2)=O^2$. Now we can use
Theorem~\ref{t:doi} and conclude that $\widetilde U^\sharp(Of)=O\widetilde U^\sharp(f)$ 
for all $f\in L^\infty_CC(X;\mu)$, hence $[O,\widetilde U]=0$ and then $[O,U]=0$.
\end{proof}

\subsection{Continuous Stochastic/Markov Semigroups.}
A \emph{continuous stochastic semigroup} on $(X;\mu)$ is a strongly continuous semigroup of 
stochastic operators on $L^1_\RR(X;\mu)$.
The \emph{infinitesimal generator} of $\{U_t\}_{t\geq 0}$ is the closed and densely defined operator
$H$ in $L^1_\RR(X;\mu)$, see Subsection~\ref{ss:fpw}.
Let $\{U_t\}_{t\geq 0}$ be a continuous
stochastic semigroup with respect to $(X;\mu)$ and $H$ its infinitesimal generator. 
Then $\{U_t^\sharp\}_{t\geq 0}$ is a $w^*$-continuous semigroup of Markov maps.
The \emph{$w^*$-infinitesimal generator} of $\{U_t^\sharp\}_{t\geq 0}$ is the $w^*$-closed, hence 
closed, and $w^*$-densely defined (but, in general, not densely defined) operator
$H^\sharp$ in $L^\infty_\RR(X;\mu)$ which, by Phillips Theorem \cite{Phillips}, can be described by
\begin{equation*}
H^\sharp f=w^*-\lim_{t\ra0+}\frac{U_tf-f}{t},\quad f\in\dom(H^\sharp),
\end{equation*}
where
\begin{equation*}
\dom(H^\sharp)=\{f\in L^\infty_\RR(X;\mu)\mid w^*-\lim_{t\ra0+}\frac{U_t^\sharp f-f}{t}
\mbox{ exists in }L^\infty_\RR(X;\mu)\}.
\end{equation*} 

The equivalence of assertions (i) and (ii) in the next theorem has been obtained in 
\cite{BaezFong}, which we now obtain as 
a consequence of Theorem~\ref{t:doi}, via Theorem~\ref{t:dms}. 
We complete their theorem with four more equivalent 
assertions in terms of infinitesimal generators and dual. The proofs
are very similar with those in Theorem~\ref{t:bartoszek2} and we omit repeating the arguments, in
particular, the equivalence of assertions (ii)$^\prime$ and (iii)$^\prime$ follows from 
Theorem~\ref{t:dual}.

\begin{theorem}\label{t:dmsd} 
Let $(X;\mu)$ be a $\sigma$-finite measure space, $\{U_t\}_{t\geq 0}$ a continuous stochastic semigroup with respect to $(X;\mu)$, $H$ its infinitesimal generator, 
and $O\in L^\infty_\RR(X;\mu)$ an observable. The following assertions are equivalent:
\begin{itemize}
\item[(i)] $[O,U_t]=0$ for all real $t\geq 0$.
\item[(ii)] For every probability distribution $g$ on $(X;\mu)$, the expected values
 $\langle O,U_tg\rangle $ and $\langle O^2,U_tg\rangle$ are constant with respect to 
 $t\geq 0$.
\item[(iii)] $[O,H]=0$, in the sense that the operator of multiplication with $O$ leaves 
$\dom(H)$ invariant and $OHg=HOg$ for all $g\in\dom(H)$.
\item[(i)$^\prime$] $[O,U_t^\sharp]=0$ for all real $t\geq 0$.
\item[(ii)$^\prime$] $U_t^\sharp(O)=O$ and $U_t^\sharp(O^2)=O^2$ for all real $t\geq 0$
\item[(iii)$^\prime$] Both $O$ and $O^2$ are in the kernel of $H^\sharp$, that is, 
$O,O^2\in\dom(H^\sharp)$ and $H^\sharp (O)=H^\sharp (O^2)=0$.
\end{itemize}
\end{theorem}

\section{Constants of Dynamical Quantum Systems}\label{s:cdqs}

We now consider the setting of dynamical quantum systems as in \cite{GoughRatiuSmolyanov}.
Notation is as in Section~\ref{s:dssm}. 
For a fixed Hilbert space $\cH$ and $A\in\cB(\cH)$ we have the \emph{left multiplication operator} 
$L_A\colon \cB_1(\cH)\ra 
\cB_1(\cH)$ defined by $L_A(T)=AT$, for all $T\in\cB_1(\cH)$, and the \emph{right multiplication
operator} $R_A\colon\cB_1(\cH)\ra\cB_1(\cH)$ defined by $R_A(T)=TA$, for all $T\in\cB_1(\cH)$.
Observe that, exactly with the same formal definition, we may have the left multiplication 
operator $L_A\colon\cB(\cH)\ra\cB(\cH)$ and, respectively, $R_A\colon \cB(\cH)\ra\cB(\cH)$. We
will not use different notation for these operators, hoping that which is which will
be clear from the context. For example, if $L_A\colon \cB_1(\cH)\ra \cB_1(\cH)$ then $L_A^\sharp
\colon\cB(\cH)\ra\cB(\cH)$ is the operator $R_A\colon\cB(\cH)\ra\cB(\cH)$. Also, considering
$M_A(T)=A^*TA$, the one-element quantum measurement operator, then $M_A=L_{A^*}R_A$.

We distinguish between the discrete quantum semigroups and 
continuous quantum semigroups.

\subsection{Discrete Quantum Semigroups}\label{ss:dqs}
We consider a quantum operation 
$\Psi\colon\cB_1(\cH)\ra\cB_1(\cH)$, that is, a trace preserving completely positive linear map.
It gives rise naturally to the discrete quantum semigroup $\{\Psi^n\}_{n\geq 0}$.
 
In order to substantiate the definition of a constant of a discrete quantum semigroup 
$\{\Psi^n\}_{n\geq 0}$ we first recall some natural definitions from quantum probability. 
Let $A$ be a bounded observable with respect to the Hilbert space $\cH$, 
that is, $A\in\cB(\cH)$ and $A=A^*$.
For any state $\rho\in\cD(\cH)$
one considers the \emph{expected value} of $A$ in the state $\rho$,
\begin{equation}\label{e:ev}
E(A;\rho)=\langle \rho, A\rangle=\tr(\rho A),
\end{equation}
the \emph{variation} of $A$ in the state $\rho$,
\begin{equation}\label{e:var}
V(A;\rho)=\langle\rho, A^2\rangle-\langle \rho,A\rangle^2=\tr(\rho A^2)-\tr(\rho A)^2,
\end{equation}
and its \emph{standard deviation},
\begin{equation}\label{e:sd}
\sigma(A;\rho)=\sqrt{\langle\rho, A^2\rangle-\langle \rho,A\rangle^2}=\sqrt{\tr(\rho A^2)-\tr(\rho A)^2}.
\end{equation}

An operator $A\in\cB(\cH)$ is called a \emph{constant} of the discrete quantum semigroup 
$\{\Psi^n\}_{n\geq 0}$, equivalentely, of $\Psi$, 
if for any state $\rho\in\cD(\cH)$, 
$\tr(\Psi^n(\rho)A)$ does not depend on $n\geq 0$, equivalently,
$\tr(\Psi(\rho)A)=\tr(\rho A)$. Clearly, $A$ is a constant of $\Psi$ if and only if for 
any $T\in\cB_1(\cH)$ we have $\tr(\Psi(T)A)=\tr(TA)$, equivalently, $\tr(T\Psi^\sharp(A))=\tr(TA)$ 
for all $T\in\cB_1(\cH)$. Consequently, $A\in\cB(\cH)$ is a constant of $\Psi$ 
if and only if $\Psi^\sharp(A)=A$, that is, $A$ is a fixed point of $\Psi^\sharp$. Formally, letting 
$\cC^\Psi$ denote the \emph{set of constants} of $\Psi$
\begin{align}\label{e:const} \cC^\Psi & =\{A\in\cB(\cH)\mid \mbox{for all }
\rho\in\cD(\cH),\ \tr(\Psi^n(\rho)A)\mbox{ does not depend on integer }n\geq 0\}\\
& =\{A\in\cB(\cH)\mid \tr(\Psi(\rho)A)=\tr(\rho A)\mbox{ for all }
\rho\in\cD(\cH)\}\nonumber \\
\intertext{we have,}
& = \{A\in\cB(\cH)\mid \Psi^\sharp(A)=A\}=\cB(\cH)^{\Psi^\sharp},\nonumber
\end{align}
where the last equality is actually the definition of $\cB(\cH)^{\Psi^\sharp}$ as 
the set of all fixed points of $\Psi^\sharp$, as in Subsection~\ref{ss:fppmcsa}. 

We have now a first Noether Type Theorem for a discrete dynamical quantum system, 
in a spirit closer to \cite{BaezFong}.

\begin{theorem}\label{t:dqs}
Let $\Psi$ be a quantum operation with respect to the Hilbert space $\cH$ and let
$A\in\cB(\cH)$. The following assertions are equivalent:
\begin{itemize}
\item[(i)] $[L_A,\Psi]=0$.
\item[(ii)] $A$ and $A^*A$ are constants of $\Psi$.
\item[(iii)] $[R_A,\Psi^\sharp]=0$.
\item[(iv)] $A$ and $A^*A$ are fixed points of $\Psi^\sharp$.
\end{itemize}
\end{theorem}

\begin{proof} (i)$\LRa$(iii) and (ii)$\LRa$(iv) are clear.

(iii)$\Ra$(iv). If $[R_A,\Psi^\sharp]=0$ then $\Psi^\sharp(SA)=\Psi^\sharp(S)A$ for all 
$S\in\cB(\cH)$. Letting $S=I$ we get $\Psi^\sharp(A)=A$ and, since $\Psi^\sharp$ is positive, 
hence selfadjoint, it follows that $\Psi^\sharp(A^*)=A^*$. Then, letting $S=A^*$ we get 
$\Psi^\sharp(A^*A)=\Psi^\sharp(A^*)A=A^*A$.

(iv)$\Ra$(iii). Assume that $\Psi^\sharp(A)=A$ and $\Psi^\sharp(A^*A)=A^*A$. Then 
$\Psi^\sharp(A^*)=A^*$ and $\Psi^\sharp(A^*A)=A^*A=\Psi^\sharp(A^*)\Psi(A)$. By
Theorem~\ref{t:choi} it follows that $\Psi^\sharp(TA)=\Psi^\sharp(T)\Psi^\sharp(A)=
\Psi^\sharp(T)A$ for all $T\in\cB(\cH)$, hence $[R_A,\Psi^\sharp]=0$.
\end{proof}

Clearly, there is a symmetric analogue of Theorem~\ref{t:dqs}.

\begin{theorem}\label{t:dqs2}
Let $\Psi$ be a quantum operation with respect to the Hilbert space $\cH$ and let
$A\in\cB(\cH)$. The following assertions are equivalent:
\begin{itemize}
\item[(i)] $[R_A,\Psi]=0$.
\item[(ii)] $A$ and $AA^*$ are constants of the discrete quantum semigroup $\{\Psi^n\}_{n\geq 0}$.
\item[(iii)] $[L_A,\Psi^\sharp]=0$.
\item[(iv)] $A$ and $AA^*$ are fixed points of $\Psi^\sharp$.
\end{itemize}
\end{theorem}

We now consider the case of a bounded observable $A\in\cB(\cH)^+$, as in \eqref{e:ev}
through \eqref{e:sd}, and reformulate Theorem~\ref{t:dqs} and its symmetric Theorem~\ref{t:dqs2}
to a noncommutative analogue of the Noether
Type Theorem as in \cite{BaezFong}, see Theorem~\ref{t:dms}.

\begin{corollary}\label{c:sandi}
Let $\Psi$ be a quantum operation with respect to the Hilbert space $\cH$ and let $A\in\cB(\cH)$
with $A=A^*$. The following assertions are equivalent.
\begin{itemize}
\item[(i)] $[L_A,\Psi]=0$.
\item[(i)$^\prime$] $[R_A,\Psi]=0$.
\item[(ii)] In any state $\rho\in\cD(\cH)$, $A$ and $A^2$ have expected values
with respect to $\Psi^n\rho$ independent of $n\geq 0$.
\item[(ii)$^{\prime}$] In any state $\rho\in\cD(\cH)$, $A$ has expected value and
standard deviation with respect to $\Psi^n\rho$ independent of $n\geq 0$.
\item[(iii)] $A$ and $A^2$ are fixed points of $\Psi^\sharp$.
\end{itemize}
\end{corollary}

In order to put the investigations in \cite{GoughRatiuSmolyanov} in a perspective closer to our 
approach, 
we now consider a scale of sets of constants of $\Psi$, more precisely, let
\begin{align}\label{e:cep2}\cC^\Psi_2 & =\{A\in\cB(\cH)\mid A,A^*A,AA^*\in\cC^\Psi\}, \\
\label{e:cepp}
\cC^\Psi_{\mathrm{p}} & =\{A\in\cB(\cH)\mid p(A,A^*)\in\cC^\Psi \mbox{ for all complex 
polynomials }p \\  \nonumber &\phantom{=\{A\in\cB(\cH)\mid p(A,A^*)\in\cC^\Psi\ } 
\mbox{ in two noncommutative variables}\},\\
\label{e:cepc}\cC^\Psi_{\mathrm{c}} & =\{A\in\cB(\cH)\mid C^*(I,A)\subseteq\cC^\Psi\},\\
\label{e:cepw}\cC^\Psi_{\mathrm{w}} & =\{A\in\cB(\cH)\mid W^*(A)\subseteq\cC^\Psi\},
\end{align}
where $C^*(I,A)$ denotes the $C^*$-algebra generated by $I$ and $A$, while $W^*(A)$ denotes
the von Neumann algebra generated by $A$. Transferring these classes in the Heisenberg picture,
we have
\begin{align}\label{e:dp2}\cC^\Psi_2 &=\{A\in\cB(\cH)\mid A,A^*A,AA^*\in\cB(\cH)^{\Psi^\sharp}\}
=\cB(\cH)^{\Psi^\sharp}_2,\\
\cC^\Psi_{\mathrm{p}}& =\{A\in\cB(\cH)\mid p(A,A^*)\in\cB(\cH)^{\Psi^\sharp}
\mbox{ for all complex polynomials }p\\
& \phantom{=\{A\in\cB(\cH)\mid p(A,A^*)\in\cC^\Psi\ }\mbox{ in two noncommutative variables}\}
\nonumber =\cB(\cH)^{\Psi^\sharp}_{\mathrm{p}}, \\
\label{e:dpc}\cC^\Psi_\mathrm{c} &=\{A\in\cB(\cH)\mid C^*(I,A)\subseteq\cB(\cH)^{\Psi^\sharp}\}
=\cB(\cH)^{\Psi^\sharp}_{\mathrm{c}},\\
\label{e:dpw}\cC^\Psi_\mathrm{w} &=\{A\in\cB(\cH)\mid W^*(A)\subseteq\cB(\cH)^{\Psi^\sharp}\}
=\cB(\cH)^{\Psi^\sharp}_{\mathrm{w}}.
\end{align}
It is easy
to see that $\cC^\Psi$ is an operator system, that is, a vector space stable under taking adjoints
and containing the identity $I$, and $w^*$-closed, hence closed with respect to the operator norm
as well. As any other operator system, $\cC^\Psi$ is linearly generated
by the set of its positive elements but, in general, not stable under multiplication, cf.\ 
\cite{Arveson1}, \cite{Arveson2}, \cite{BratelliJorgensenKishimotoWerner}.

\begin{theorem}\label{t:faq} Let $\Psi$ be a quantum operation with respect to a Hilbert space 
$\cH$.
\begin{itemize}
\item[(a)] We always have
$\cC^\Psi_2=\cC^\Psi_{\mathrm{p}}=\cC^\Psi_{\mathrm{c}}=\cC^\Psi_{\mathrm{w}}$ and this
set is a von Neumann algebra.
\item[(b)] The following assertions are equivalent:
\begin{itemize}
\item[(i)] $\cC^\Psi$ is stable under multiplication.
\item[(ii)] $\cC^\Psi=\cC^\Psi_2$.
\item[(iii)] $\cC^\Psi$ is a $C^*$-algebra.
\item[(iv)] $\cC^\Psi$ is a von Neumann algebra.
\end{itemize}
\end{itemize}
\end{theorem}

\begin{proof} (a) Clearly,
$\cC^\Psi_2\supseteq\cC^\Psi_{\mathrm{p}}\supseteq
\cC^\Psi_{\mathrm{c}}\supseteq\cC^\Psi_{\mathrm{w}}.$
Due to the density of the set of all operators 
$p(A,A^*)$, where $p$ is an arbitrary complex polynomial in two noncommutative variables, in 
$C^*(I,A)$, the $w^*$-density of $C^*(I,A)$ in $W^*(A)$, as well as the continuity and 
$w^*$-continuity of the map $A\mapsto \tr(\Psi(\rho)A)$, we have the equality
$\cC^\Psi_{\mathrm{p}}=\cC^\Psi_{\mathrm{c}}=\cC^\Psi_{\mathrm{w}}.$
On the other hand, using the dual representations as in \eqref{e:dp2} and \eqref{e:dpc}, from
Theorem~\ref{t:doi} we obtain $\cC^\Psi_2=\cC^\Psi_{\mathrm{c}}$.

In order to prove that this set is a von Neumann algebra, it is preferable to use its representation in 
the Heisenberg picture as $\cB(\cH)^{\Psi^\sharp}_2$, see \eqref{e:dp2}. 
Since $\Psi^\sharp$ is positive it is 
selfadjoint, hence $\cB(\cH)^{\Psi^\sharp}_2$ is stable under taking the involution $A\mapsto A^*$.
If $A,B\in\cB(\cH)^{\Psi^\sharp}_2$, by Theorem~\ref{t:doi} we have
\begin{equation}\label{e:psab}\Psi^\sharp(AB)=A\Psi^\sharp(B)=AB,
\end{equation} hence $\cB(\cH)^{\Psi^\sharp}_2$ is stable under multiplication. On the other hand,
\begin{align*}
\Psi^\sharp((A+B)^*(A+B)) & = \Psi^\sharp(A^*A+A^*B+B^*A+B^*A) \\
& = \Psi^\sharp(A^*A)+\Psi^\sharp(A^*B)+\Psi^\sharp(B^*A)+\Psi^\sharp(B^*A) \\
\intertext{hence, taking into account of \eqref{e:psab},}
& = A^*A+A^*B+B^*A+B^*A=(A+B)^*(A+B).
\end{align*}
Similarly we prove that $(A+B)(A+B)^*$ is a fixed point of $\Psi^\sharp$. Since, clearly $A+B$ is a
fixed point of $\Psi^\sharp$, 
it follows that $\cB(\cH)^{\Psi^\sharp}_2$ is stable under addition as well. On
the other hand, since $\Psi^\sharp$ is $w^*$-continuous, it follows that 
$\cB(\cH)^{\Psi^\sharp}_2$ is a von Neumann algebra.

(b) This is actually a reformulation of Lemma~2.2 in \cite{AriasGheondeaGudder}.
\end{proof}

\begin{remarks}\label{r:cexq} (1) There are special 
situations when one, hence all, of the assertions (i) through (iv) in Theorem~\ref{t:faq} hold(s) 
automatically: for example, if there exists a state $\omega$ of $\cB(\cH)$ that is \emph{faithful}, 
in the sense that $\omega(C^*C)=0$ implies $C=0$, and \emph{invariant} 
under $\Psi^\sharp$, that is, $\omega\circ\Psi^\sharp=\omega$, e.g.\
see Theorem~2.3 in \cite{AriasGheondeaGudder}. 

(2) However, as Theorem~4.2 in 
\cite{AriasGheondeaGudder} shows, there exist a L\"uders operation $\Phi_\cA$ with 
$\cA=\{A_1,A_2,A_3,A_4,A_5\}$, with notation as in \eqref{e:pac}, on $\cH=\ell_2(\FF_2)$,  
where $\FF_2$ denotes the free group on two generators, and an operator 
$B\in\cB(\cH)^+$ with $\Phi_\cA^\sharp(B)=B$ but $\Phi_\cA^\sharp(B^2)\neq B^2$. 
Consequently, both conditions
in Theorem~\ref{t:dqs} item (iv) are necessary, in general, and the set of constants $\cC^{\Psi}$, for
some quantum operations $\Psi$, is not stable under multiplication, 
equivalently, it is not a von Neumann algebra, see Theorem~\ref{t:faq}.(b). This answers
the question raised in Remark~\ref{r:general}.(3) as well.
\end{remarks}

\subsection{Continuous Quantum Semigroups.}\label{ss:cqs}
A family indexed on the set of nonnegative real numbers $\mathbf{\Psi}=\{\Psi_t\}_{t\geq 0}$ is
called a \emph{dynamical quantum system}, sometimes called a 
\emph{dynamical quantum stochastic system}, with respect to a Hilbert space $\cH$, 
if it is a strongly continuous semigroup of quantum operations 
$\Psi_t\colon\cB_1(\cH)\ra\cB_1(\cH)$, $t\geq 0$.
For a dynamical quantum system $\mathbf{\Psi}$, we consider its infinitesimal generator 
$\psi$, see Subsection~\ref{ss:fpw} 
for the general setting. 
This definition makes a representation of the dynamical quantum system 
$\mathbf{\Psi}$ into the Schr\"odinger picture. Transferring a dynamical quantum system 
$\mathbf{\Psi}$ into the Heisenberg picture, we get its dual, usually called 
\emph{dynamical quantum Markov system}, 
$\mathbf{\Psi}^\sharp=\{\Psi^\sharp_t\}_{t\geq 0}$ which is a $w^*$-continuous 
one-parameter semigroup of $w^*$-continuous, unital, completely positive linear maps 
$\Psi_t^\sharp\colon \cB(\cH)\ra\cB(\cH)$ to which one associates its $w^*$-infinitesimal generator
$\psi^\sharp$, as in \eqref{e:genp} and \eqref{e:domgenp}. Here, an important issue is that 
by Phillips's Theorem \cite{Phillips}, $\psi^\sharp$ is indeed the dual of $\psi$. 

Note that our
definitions are more general than those usually considered
in most mathematical models of quantum open systems, 
e.g.\ see \cite{GoughRatiuSmolyanov}, \cite{FagnolaRebolledo} 
and the rich bibliography cited there, which instead of strong continuity 
requires the (operator) norm continuity, that is, the mapping 
$\RR_+\ni t\mapsto \Psi_t\in\cL(\cB_1(\cH))$ should be continuous
with respect to the operator norm of $\cL(\cB_1(\cH))$.

An operator $A\in\cB(\cH)$ is called a \emph{constant} of the dynamical quantum system 
$\mathbf{\Psi}=\{\Psi_t\}_{t\geq 0}$,
if, for any density operator $\rho\in\cD(\cH)$, 
$\tr(\Psi_t(\rho)A)$ does not depend on $t\geq 0$, equivalently, $\tr(\Psi_t(\rho)A)=\tr(\rho A)$ 
for all $t\geq 0$. Clearly, $A$ is a constant of $\mathbf{\Psi}$ if and only if for 
any $T\in\cB_1(\cH)$ we have $\tr(\Psi_t(T)A)=\tr(TA)$ for all $t\geq 0$, 
equivalently, $\tr(T\Psi_t^\sharp(A))=\tr(TA)$ 
for all $T\in\cB_1(\cH)$ and all $t\geq 0$. Consequently, $A\in\cB(\cH)$ is a constant of $\Psi$ 
if and only if $\Psi^\sharp_t(A)=A$ for all $t\geq 0$, that is, $A$ is a fixed point of $\Psi_t^\sharp$
for all $t\geq 0$. Formally, letting $\cC^{\mathbf{\Psi}}$ denote the \emph{set of constants} of 
$\mathbf{\Psi}$
\begin{align}\label{e:const2} \cC^{\mathbf{\Psi}} & =\{A\in\cB(\cH)\mid \mbox{for all }
\rho\in\cD(\cH),\ \tr(\Psi_t(\rho)A)\mbox{ does not depend on real }t\geq 0\}\\
& =\{A\in\cB(\cH)\mid \tr(\Psi_t(\rho)A)=\tr(\rho A)\mbox{ for all }
\rho\in\cD(\cH)\mbox{ and all }t\geq 0\}\nonumber \\
& = \{A\in\cB(\cH)\mid \Psi_t^\sharp(A)=A\mbox{ for all }t\geq 0\}=
\cB(\cH)^{\mathbf{\Psi}^\sharp},\nonumber
\end{align}
where the last equality is actually the definition of $\cB(\cH)^{\mathbf{\Psi}^\sharp}$ 
as the set of all joint fixed points of $\Psi_t^\sharp$, $t\geq 0$. In addition, as a consequence
of Theorem~\ref{t:dual}, we have
\begin{equation}\label{e:kernel}
\cC^{\mathbf{\Psi}}=\cB(\cH)^{\mathbf{\Psi}^\sharp}=\ker(\psi^\sharp)=\{T\in\cB(\cH)\mid T\in
\dom(\psi^\sharp),\ \psi^\sharp(T)=0\}.
\end{equation}

The next theorem is the continuous variant of Theorem~\ref{t:dqs}.
\begin{theorem}\label{t:dqsc}
Let $\mathbf{\Psi}=\{\Psi_t\}_{t\geq 0}$ be a dynamical quantum stochastic system 
with respect to the Hilbert space $\cH$, let $\psi$ denote its infinitesimal generator, 
and let $A\in\cB(\cH)$. The following assertions are equivalent:
\begin{itemize}
\item[(i)] $[L_A,\Psi_t]=0$ for all $t\geq 0$.
\item[(ii)] $A$ and $A^*A$ are constants of $\mathbf{\Psi}$.
\item[(iii)] $[R_A,\Psi_t^\sharp]=0$ for all $t\geq 0$.
\item[(iv)] $A$ and $A^*A$ are joint fixed points of $\Psi_t^\sharp$ for all $t\geq 0$.
\item[(v)] $[L_A,\psi]=0$, that is, $L_A$ leaves $\dom(\psi)$ invariant and $A\psi(T)=\psi(AT)$ for
all $T\in\dom(\psi)$.
\item[(vi)] $A,A^*A\in\ker(\psi^\sharp)$, that is, $A,A^*A\in\dom(\psi^\sharp)$ and $\psi^\sharp(A)=
\psi^\sharp(A^*A)=0$.
\end{itemize}
\end{theorem}
The equivalence of assertions (i)--(iv) follows from Theorem~\ref{t:dqs}, while their equivalence 
with assertions (v) and (vi) can be proven similarly as in the proof of 
Theorem~\ref{t:bartoszek2} and we do not repeat the arguments. In particular, the equivalence of 
assertions (iv) and (vi) in the previous theorem follows from Theorem~\ref{t:dual}, 
but this is a more general fact, see \eqref{e:kernel}.

There is a symmetric variant to Theorem~\ref{t:dqsc}, 
in which $L_A$ and $R_A$ are interchanged and, correspondingly,
$A^*A$ and $AA^*$ are interchanged, the continuous dynamical system variant of 
Theorem~\ref{t:dqs2}. We leave the reader to formulate it.

In case of a bounded observable $A\in\cB(\cH)^+$, with expected value, variation, and standard 
deviation to an arbitrary state $\rho\in\cD(\cH)$ as in \eqref{e:ev}
through \eqref{e:sd}, Theorem~\ref{t:dqsc} can be reformulated
to a noncommutative analogue of the Noether
Type Theorem as in \cite{BaezFong}, see Theorem~\ref{t:dms}.

\begin{corollary}\label{c:dqsc}
Let $\mathbf{\Psi}=\{\Psi_t\}_{t\geq 0}$ be a dynamical quantum stochastic system 
with respect to the Hilbert space $\cH$, let $\psi$ denote its infinitesimal generator, 
and let $A\in\cB(\cH)$, $A=A^*$, be a bounded observable. 
The following assertions are equivalent:
\begin{itemize}
\item[(i)] $[L_A,\Psi_t]=0$ for all $t\geq 0$.
\item[(i)$^\prime$] $[R_A,\Psi_t]=0$ for all $t\geq 0$.
\item[(ii)] In any state $\rho\in\cD(\cH)$, $A$ and $A^2$ have expected values with respect to 
$\Psi_t$ independent of $t\geq 0$.
\item[(ii)$^\prime$] In any state $\rho\in\cD(\cH)$, $A$ has expected value and standard deviation 
with respect to $\Psi_t$ independent of $t\geq 0$.
\item[(iii)] $[R_A,\Psi_t^\sharp]=0$ for all $t\geq 0$.
\item[(iii)$^\prime$] $[L_A,\Psi_t^\sharp]=0$ for all $t\geq 0$.
\item[(iv)] $A$ and $A^2$ are joint fixed points of $\Psi_t^\sharp$ for all $t\geq 0$.
\item[(v)] $[L_A,\psi]=0$, that is, $L_A$ leaves $\dom(\psi)$ invariant and $A\psi(T)=\psi(AT)$ for
all $T\in\dom(\psi)$.
\item[(v$^\prime$)] $[R_A,\psi]=0$, that is, $R_A$ leaves $\dom(\psi)$ invariant and 
$\psi(T)A=\psi(TA)$ for all $T\in\dom(\psi)$.
\item[(vi)] $A,A^2\in\ker(\psi^\sharp)$, that is, $A,A^2\in\dom(\psi^\sharp)$ and $\psi^\sharp(A)=
\psi^\sharp(A^2)=0$.
\end{itemize}
\end{corollary}

On the other hand, as in \eqref{e:cep2}--\eqref{e:cepw}, 
we have the joint versions of the scale of sets of constants 
\begin{equation*}
\cC^{\mathbf{\Psi}}_{\mathrm{w}}\subseteq\cC^{\mathbf{\Psi}}_{\mathrm{c}}\subseteq
\cC^{\mathbf{\Psi}}_{\mathrm{p}}\subseteq\cC^{\mathbf{\Psi}}_{2}\subseteq\cC^{\mathbf{\Psi}},
\end{equation*} 
in the Schr\"odinger picture, 
more precisely,
\begin{equation}\cC^{\mathbf{\Psi}}_\bullet=\bigcap_{t\geq 0}\cC^{\Psi_t}_\bullet,\mbox{ where }
\bullet=2,\mathrm{p},\mathrm{c},\mathrm{w},
\end{equation}
and, as in \eqref{e:dp2}--\eqref{e:dpw}, 
the sets of joint fixed points 
\begin{equation*}
\cB(\cH)^{\mathbf{\Psi}^\sharp}_{\mathrm{w}}\subseteq
\cB(\cH)^{\mathbf{\Psi}^\sharp}_{\mathrm{c}}\subseteq\cB(\cH)^{\mathbf{\Psi}^\sharp}_{\mathrm{p}}
\subseteq\cB(\cH)^{\mathbf{\Psi}^\sharp}_2\subseteq\cB(\cH)^{\mathbf{\Psi}^\sharp},
\end{equation*} 
in the Heisenberg picture,
\begin{equation}
\cB(\cH)^{\mathbf{\Psi}^\sharp}_\bullet=\bigcap_{t\geq 0}\cB(\cH)^{\Psi_t^\sharp}_\bullet,
\mbox{ where }\bullet=2,\mathrm{p},\mathrm{c},\mathrm{w}.
\end{equation}
$\cC^{\mathbf{\Psi}}=\cB(\cH)^{\mathbf{\Psi}}$ 
is a $w^*$-closed operator system and $w^*$-closed, hence closed with 
respect to the operator norm on $\cB(\cH)$ as well, linearly generated
by the set of its positive elements but, in general, not stable under multiplication. The following
theorem is a consequence of Theorem~\ref{t:faq}.

\begin{theorem}\label{t:faqs} Let $\mathbf{\Psi}$ be a dynamical quantum system with respect
to the Hilbert space $\cH$.
\begin{itemize}
\item[(a)] For any dynamical quantum system $\mathbf{\Psi}$ we have
$\cC^{\mathbf{\Psi}}_2=\cC^{\mathbf{\Psi}}_{\mathrm{p}}=\cC^{\mathbf{\Psi}}_{\mathrm{c}}
=\cC^{\mathbf{\Psi}}_{\mathrm{w}}$ and this set is a von Neumann algebra.
\item[(b)] The following assertions are equivalent:
\begin{itemize}
\item[(i)] $\cC^{\mathbf{\Psi}}$ is stable under multiplication.
\item[(ii)] $\cC^{\mathbf{\Psi}}=\cC^{\mathbf{\Psi}}_2$.
\item[(iii)] $\cC^{\mathbf{\Psi}}$ is a $C^*$-algebra.
\item[(iv)] $\cC^{\mathbf{\Psi}}$ is a von Neumann algebra.
\end{itemize}
\end{itemize}
\end{theorem}

\begin{remarks}\label{r:cexc}
 (a) The main theorem in \cite{GoughRatiuSmolyanov} states that, for a dynamical
quantum (stochastic) system $\mathbf{\Psi}$ under two additional constraints, namely, that the 
semigroup is (operator) norm continuous and that there exists a stationary strictly positive density 
operator, that is,
there exists $\rho\in\cB_1(\cH)^+$ that is strictly positive and such that $\Psi_t(\rho)=\rho$ for
all real $t\geq 0$, then $\cC^{\mathbf{\Psi}^\sharp}=\cB(\cH)^{\mathbf{\Psi}^\sharp}$ is a von
Neumann algebra. This theorem remains true under the general assumption 
that the semigroup $\mathbf{\Psi}$ is strongly continuous: we use Theorem~\ref{t:faqs} while
the existence of a stationary strictly positive density operator $\rho$
implies the existence of a normal faithful stationary state $\omega(T)=\tr(\rho T)$, $T\in\cB(\cH)$,
and then Theorem~2.3 in \cite{AriasGheondeaGudder}.

(b) In case the dynamical quantum system $\mathbf{\Psi}$ is (operator) norm continuous,
the infinitesimal generator $\psi$ is bounded and,
by a result of G.~Lindblad \cite{Lindblad} (and, in the finite dimensional case, of V.~Gorini, 
A.~Kossakowski, and E.C.G.~Sudarshan \cite{GoriniKossakowskiSudarshan}), it
takes the form
\begin{equation}\label{e:igs}
\psi(S)=\sum_{k=1}^\infty
(L_kSL_k^*-\frac{1}{2} SL_k^*L_k-\frac{1}{2}L_k^*L_kS)+\iac[S,H],\quad 
S\in\cB_1(\cH),
\end{equation}
for a collection of operators $L_k\in\cB(\cH)$, $k=1,2,\dots$, and a selfadjoint operator 
$H\in\cB(\cH)$. It is easy to see that it's adjoint, which is the infinitesimal generator of
the dual quantum Markov semigroup $\{\Psi_t^\sharp\}_{t\geq 0}$, is
\begin{equation}\label{e:igsp}
\psi^\sharp(T)=\sum_{k=1}^\infty (L_k^*TL_k-\frac{1}{2}L_k^*L_kT-\frac{1}{2}TL_k^*L_k)-\iac[T,H],
\quad T\in\cB(\cH).
\end{equation}
Consequently, using \eqref{e:kernel}, it follows that the constants of $\mathbf{\Psi}$ are exactly
the solutions $T\in\cB(\cH)$ of the equation
\begin{equation}\label{e:duaeq}
\sum_{k=1}^\infty (L_k^*TL_k-\frac{1}{2}L_k^*L_kT-\frac{1}{2}TL_k^*L_k)-\iac[T,H]=0,
\end{equation}
which is an operator Riccati equation.

(c) In case the dynamical quantum system $\mathbf{\Psi}$ is (operator) norm continuous, 
hence \eqref{e:igs}
and \eqref{e:igsp} hold, and $\mathbf{\Psi}$ has  a stationary strictly positive density operator, 
it is proven in \cite{GoughRatiuSmolyanov} that the set
$\cC^{\mathbf{\Psi}}_{\mathrm{w}}$ coincides with the commutant of 
$\{H,L_k,L_k^*\mid k=1,2,\ldots\}^\prime$, in particular, it is a von Neumann algebra.
\end{remarks}

Finally, we are in a position to approach the following question: 
to which extent are the latter conditions on
$A^*A$ or $AA^*$, as in Theorem~\ref{t:dqsc}.(ii), and the latter condition on $A^2$ as in 
Corollary~\ref{c:dqsc}, really necessary? Note that a positive answer to 
this question will answer the similar question asked for the more general case of dynamical 
stochastic systems as in Subsection~\ref{ss:sopssm}, see Remark~\ref{r:uta}.(b).

\begin{example}\label{ex:fg}
As in \cite{AriasGheondeaGudder},
let $\FF_2$ denote the free group on two generators $g_1$ and $g_2$, and let
$\ell^2(\FF_2)$ denote the Hilbert space of all square summable functions $f\colon\FF_2\ra\CC$.
In $\ell^2(\FF_2)$ a canonical orthonormal basis is made up by $\{\delta_x\}_{x\in\FF_2}$, where 
$\delta_x(y)=0$ for all $y\in\FF_2$, $y\neq x$, and $\delta_x(x)=1$. 
Since $\FF_2$ is infinitely 
countable it follows that $\ell^2(\FF_2)$ is infinite dimensional and separable. 
Let $U_j\in\cB(\ell_2(\FF_2))$ denote the unitary operators 
$U_j\delta_x=\delta_{g_jx}$, $x\in\FF_2$ and $j=1,2$. 

We consider the linear bounded operator
$\psi\colon\cB_1(\ell^2(\FF_2))\ra \cB_1(\ell^2(\FF_2))$ defined by
\begin{equation} \psi(S)=U_1SU_1^*+U_2SU_2^*-2S,\quad S\in\cB_1(\ell^2(\FF_2)),
\end{equation}
and then let
\begin{equation}
\Psi_t(S)=\exp(t\psi(S)),\quad S\in\cB_1(\ell^2(\FF_2)),\ t\geq 0.
\end{equation}
From \cite{Lindblad},  see Remark~\ref{r:cexc}.(b),
it follows that $\mathbf{\Psi}=\{\Psi_t\}_{t\geq 0}$ is a (operator) norm continuous semigroup
of quantum operations with respect to $\ell^2(\FF_2)$.

Also, let $L(\FF_2)=W^*(U_1,U_2)$ denote the group von Neumann algebra of $\FF_2$. We 
observe, e.g.\ by means of \eqref{e:duaeq}, that the commutant von Neumann algebra 
$L(\FF_2)^\prime$ is included in the set of constants $\cC^{\mathbf{\Psi}}$.
\end{example}

\begin{lemma}\label{l:fg}
Let $\mathbf{\Psi}$ be the dynamical quantum system as in Example~\ref{ex:fg}.
Then, $\cC^{\mathbf{\Psi}}$ is stable under multiplication if and only if it coincides with
$L(\FF_2)^\prime$.
\end{lemma}

\begin{proof} It is sufficient to prove that, if $\cC^{\mathbf{\Psi}}$ is stable under multiplication then
it coincides with $L(\FF_2)^\prime$.
To see this, assume that $\cC^{\mathbf{\Psi}}$ is stable under multiplication
hence, by Theorem~\ref{t:faqs}.(b), it is a von Neumann algebra. By 
Remark~\ref{r:cexc}.(b), it follows that for any orthogonal projection $E\in\cC^{\mathbf{\Psi}}$
equation \eqref{e:duaeq} holds which, in our special case, is
\begin{equation} U_1^*EU_1+U_2^*EU_2=2E.
\end{equation}
Consequently, for each vector $h\in\ell^2(\FF_2)$ that lies in the range of $E$ we have
\begin{equation*}
\|EU_1h\|^2+\|EU_2h\|^2=\langle U_1^*EU_1h,h\rangle+\langle U_2^*EU_2h,h\rangle=
2\langle Eh,h\rangle =2\|h\|^2,
\end{equation*}
from which, after a moment of thought, we see that $U_jh$ should lie in the range of 
$E$ for $j=1,2$. We have shown that $U_j$ leaves the range of $E$ invariant, $j=1,2$. 
Since the same is true for the range of $I-E$, it follows that $U_j$ commutes with all orthogonal
projections in the von Neumann algebra $\cC^{\mathbf{\Psi}}$, hence 
$\cC^{\mathbf{\Psi}}\subseteq \{U_1,U_1^*,U_2,U_2^*\}^\prime=L(\FF_2)^\prime$. The converse
inclusion was observed at the end of Example~\ref{ex:fg}.
\end{proof}

During the proof of the next theorem we use terminology as in Subsection~\ref{ss:etvna}.

\begin{theorem}\label{t:dqscc}
On any infinite dimensional separable Hilbert space $\cH$, there exists a (operator) 
norm continuous
semigroup of quantum operations $\mathbf{\Phi}=\{\Phi_t\}_{t\geq 0}$ with respect to $\cH$, for
which:
\begin{itemize}
\item[(a)] The set of constants $\cC^{\mathbf{\Phi}}$ is not a von Neumann algebra, equivalently,
it is not stable under multiplication.
\item[(b)] There exists $A\in\cB(\cH)^+$ which is a constant of $\mathbf{\Phi}$ but $A^2$ is 
not.
\end{itemize}
\end{theorem}

\begin{proof} We first show that the (operator) 
norm continuous dynamical quantum system $\mathbf{\Psi}$ 
as in Example~\ref{ex:fg} has all the required properties.  To this end, it is sufficient to prove 
assertion (b), then assertion (a) will follow from Theorem~\ref{t:faqs}.
By a classical result of J.~Hakeda and J.~Tomiyama 
\cite{HakedaTomiyama}, a von Neumann algebra $\cM$ is injective if and only if its commutant 
$\cM^\prime$ is injective. By another classical result of J.T.~Schwartz \cite{Schwartz}, see also
J.~Tomiyama \cite{Tomiyama2}, the
von Neumann algebra $L(\FF_2)$ is not injective hence, its commutant $L(\FF_2)^\prime$
is not injective either. 
Consequently, by Lemma~\ref{l:fg} and Theorem~\ref{t:inj}, 
the set of joint
fixed points $\cB(\ell^2(\FF_2))^{\mathbf{\Psi^\sharp}}=\cC^{\mathbf{\Psi}}$ is strictly larger
than the joint bimodule set $\cI(\mathbf{\Psi^\sharp})$. Since 
$\cB(\ell^2(\FF_2))^{\mathbf{\Psi^\sharp}}$ is an operator system, hence linearly generated by
its positive cone, there exits 
$A\in\cB(\ell^2(\FF_2))^{\mathbf{\Psi^\sharp}}
\setminus\cI(\mathbf{\Phi})$ with $A\geq 0$.
In view of Theorem~\ref{t:choi}, this implies $A^2\not\in\cB(\ell^2(\FF_2))^{\mathbf{\Psi^\sharp}}$.

In general, if $\cH$ is an infinite dimensional separable Hilbert space, then there exists a unitary
operator $U\colon\ell^2(\FF_2)\ra\cH$ and let $\Phi_t=U^*\Psi_tU$, for all real $t\geq 0$. Then
$\mathbf{\Phi}=\{\Phi_t\}_{t\geq 0}$ has all the required properties.
\end{proof}

Theorem~\ref{t:dqscc} answers, in the negative, 
also the question on whether the condition that $A^2$ is a joint 
fixed point, as in Theorem~\ref{t:bartoszek2}.(i), 
is a consequence of the condition that $A$ is a joint fixed point.

\section*{Appendix: A Proof of Theorem~\ref{t:choi}}\label{s:aa}
%
\setcounter{equation}{0}
\renewcommand\theequation{a.\arabic{equation}} 

Let $\Phi\colon\cA\ra\cB$ be a contractive completely positive map hence, by the Russ-Dye 
Theorem, $\|\Phi(e)\|=\|\Phi\|\leq 1$. 
By the Gelfand--Naimark Theorem, $\cB$
can be faithfully represented as a $C^*$-subalgebra of $\cB(\cH)$, for some Hilbert space $\cH$, 
hence $\Psi$ can be considered as a completely positive map 
$\Phi\colon \cA\ra\cB(\cH)$. Then, by
Stinespring Theorem \cite{Stinespring}, 
there exists a triple $(\cK;\pi;V)$, where $\cK$ is a Hilbert space, 
$\pi\colon\cA\ra \cB(\cK)$ is a $*$-representation, $V\colon \cH\ra\cK$ is a bounded linear operator, 
$\phi(a)=V^*\pi(a)V$ for all $a\in \cA$, and such that $\pi(\cA)\cH$ is a total subset of $\cK$. We 
briefly recall the original construction \cite{Stinespring}. $\cK$
is defined as the factorization and
completion to a Hilbert space of the vector space $\cA\otimes\cH$ with respect to the inner product 
$\langle\cdot,\cdot\rangle_\cK$, defined on elementary tensors by
$ \langle a\otimes h,b\otimes k\rangle_\cK=\langle \Phi(b^*a)h,k\rangle_\cH$, for all
$a,b\in\cA,\ h,k\in\cH$, and then extended by linearity. The representation $\pi$ is defined on 
elementary tensors by $\pi(a)(b\otimes h)=ab\otimes h$, for all $a,b\in\cA$, $h\in\cH$,
and then extended by linearity, while the operator $V$ is defined by
$Vh=\Phi(e)\otimes h$, $h\in\cH$,
and, since $\|\Phi(e)\|\leq 1$ then $\|V\|\leq 1$, equivalent to $I_\cH-V^*V\geq 0$ and to
$I_\cK-VV^*\geq0$.

(1) \emph{The Schwarz Inequality}. For any $a\in\cA$
we have
\begin{align*}
\Phi(a^*a)-\Phi(a)^*\Phi(a) & = V^*\pi(a^*a)V-V^*\pi(a)^*VV^*\pi(a)V \\
& =V^*\pi(a)^*(I_\cK-VV^*)\pi(a)V\geq 0,
\end{align*}
hence,
$\Phi(a)^*\Phi(a)\leq \Phi(a^*a)$.

(2) \emph{The Multiplicativity Property}. For arbitrary but fixed $a\in\cA$, we have
\begin{align*}
\|\Phi(a)^*\Phi(a)-\Phi(a^*a)\| & =\|V^*\pi(a)^*(I_\cK-VV^*)\pi(a)V\| \\
& = \|V^*\pi(a)^*(I_\cK-VV^*)^{1/2}(I_\cK-VV^*)^{1/2}\pi(a)V\| \\
& = \|(I_\cK-VV^*)^{1/2}\pi(a)V\|^2 \\
\end{align*}
hence, 
$\Phi(a)^*\Phi(a)=\Phi(a^*a)$, if and only if $(I_\cK-VV^*)^{1/2}\pi(a)V=0$ and, taking account
that, for any $b\in\cA$,
\begin{equation*}\label{e:fani}
\Phi(ba)-\Phi(b)\Phi(a)=V^*\pi(b)(I_\cK-VV^*)^{1/2}(I_\cK-VV^*)^{1/2}\pi(a)V,
\end{equation*}
it follows that 
\begin{equation*}\label{e:fias}
\Phi(a)^*\Phi(a)=\Phi(a^*a)\mbox{ if and only if }
\Phi(b)\Phi(a)=\Phi(ba),\mbox{ for all }b\in\cA.\end{equation*}
In a similar fashion the latter equivalence follows.

(3) \emph{$\cM_\Phi$ is the largest $C^*$-subalgebra $\cC$ of $\cA$ such that 
$\Phi|_\cC\colon\cC\ra \cB$ is a $*$-homomorphism}. Clearly, $\cM_\Phi$ is stable under
taking adjoints. From (2) it follows that $\cM_\Phi$ is stable under multiplication. 
Let $a,b\in\cM_\Phi$. Then
\begin{align*} \Phi((a+b)^*(a+b)) &  = \Phi(a^*a+a^*b+b^*a+b^*b)  = \Phi(a^*a)+\Phi(a^*b)+\Phi(b^*a)+\Phi(b^*b) \\
& = \Phi(a)^*\Phi(a)+\Phi(a)^*\Phi(b)+\Phi(b)^*\Phi(a)+\Phi(b)^*\Phi(b) = \Phi(a+b)^*\Phi(a+b).
\end{align*}
Similarly we prove that $\Phi((a+b)(a+b)^*)=\Phi(a+b)\Phi(a+b)^*$,
hence $\cM_\Phi$ is stable under addition as well and, consequently, a $C^*$-subalgebra of $\cA$.

Clearly, $\Phi|_{\cM_\Phi}$ is a $*$-homomorphism. Let $\cC$ be an arbitrary $C^*$-subalgebra of 
$\cA$ such that $\Phi|_\cC$ is a $*$-homomorphism. Then, for any $a\in\cC$ we have 
$\Phi(a^*a)=\Phi(a)^*\Phi(a)$ and $\Phi(aa^*)=\Phi(a)\Phi(a)^*$, hence $a\in\cM_\Phi$.

\end{document}